\documentclass[12pt]{amsart}
\usepackage{amscd,amsmath,amssymb,amsfonts,verbatim}
\usepackage[cmtip, all]{xy}
\usepackage[a4paper, left=2.55cm, right=2.55cm, top = 2.2cm, bottom = 2cm ]{geometry}

\newtheorem{thm}{Theorem}[subsection]
\newtheorem{prop}[thm]{Proposition}
\newtheorem{lem}[thm]{Lemma}
\newtheorem{cor}[thm]{Corollary}

\theoremstyle{definition}

\newtheorem{defn}[thm]{Definition}
\theoremstyle{remark}
\newtheorem{remk}[thm]{Remark}
\newtheorem{remks}[thm]{Remarks}

\newtheorem{exm}[thm]{Example}
\newtheorem{exms}[thm]{Examples}
\newtheorem{notat}[thm]{Notation}
\numberwithin{equation}{subsection}

{\hfill$\square$\end{defn}}
{\hfill$\square$\end{remk}}
{\hfill$\square$\end{remks}}
{\hfill$\square$\end{exm}}
{\hfill$\square$\end{exms}}
{\hfill$\square$\end{notat}}

\newcommand{\CH}{{\rm CH}}

\newcommand{\codim}{{\rm codim}}

\newcommand{\Spec}{{\rm Spec \,}}

\newcommand{\Tr}{{\rm Tr}}

\newcommand{\ord}{{\rm ord}}

\newcommand{\ds}{{/\kern-3pt/}}

\newcommand{\res}{{\operatorname{res}}}

\renewcommand{\log}{{\operatorname{log}}}

\newcommand{\Proj}{{\operatorname{Proj}}}

\newcommand{\sgn}{{\operatorname{\rm sgn}}}

\newcommand{\un}{\underline}
\newcommand{\ov}{\overline}

\renewcommand{\dim}{\text{\rm dim}}

\newcommand{\tuborg}{\left\{\begin{array}{ll}}
\newcommand{\sluttuborg}{\end{array}\right.}

\newcommand{\sfs}{{\rm sfs}}

\renewcommand{\mod}{ {\rm \ mod \ } }

\begin{document}
\title{Motivic cohomology of fat points in Milnor range}
\author{Jinhyun Park and Sinan \"Unver}
\address{Department of Mathematical Sciences, KAIST, 291 Daehak-ro Yuseong-gu, Daejeon, 34141, Republic of Korea (South)}
\email{jinhyun@mathsci.kaist.ac.kr; jinhyun@kaist.edu}
\address{Department of Mathematics, Ko\c{c} University, Rumelifeneri Yolu, 34450 Sar{\i}yer-\.{I}stanbul, Turkey}
\email{sunver@ku.edu.tr}




\begin{abstract}
We introduce a new algebraic-cycle model for the motivic cohomology theory of truncated polynomials $k[t]/(t^m)$ in one variable. This approach uses ideas from the deformation theory and non-archimedean analysis, and is distinct from the approaches via cycles with modulus. We prove that the groups in the Milnor range give the Milnor $K$-groups of $k[t]/(t^m)$, when the base field is of characteristic $0$. Its relative part is the sum of the absolute K\"ahler differential forms.
\end{abstract}

\maketitle

\section{Introduction}

The objective of this paper is to present a new algebraic-cycle model for the motivic cohomology theory of some singular $k$-schemes and to compute its simplest case concretely, to justify the model.

Bloch's higher Chow groups \cite{Bloch HC} of smooth $k$-schemes give the correct motivic cohomology groups as shown by Voevodsky \cite{Voevodsky}, but they fail to do so for singular $k$-schemes. For instance, the motivic cohomology groups are expected to be part of a conjectural Atiyah-Hirzebruch type spectral sequence that converges to higher algebraic $K$-groups \cite{Quillen K} of Quillen. Here the $K$-groups do detect the difference of a scheme $X$ and and its reduced scheme $X_{\rm red}$ (see e.g. \cite{vdK}), while the higher Chow groups do not distinguish $X$ from $X_{\rm red}$. 

The additive higher Chow groups, initiated by Bloch-Esnault \cite{BE2}, were in a sense born as a way to complement this issue for non-reduced schemes. This approach developed further by e.g. \cite{KL}, \cite{KP crys}, \cite{KP Doc}, \cite{Park2}, \cite{Park}, \cite{R}, has had several successful aspects; for instance, they provide understanding of Witt vectors, de Rham-Witt complexes and crystalline cohomology via algebraic cycles based on the moving lemma of \cite{Kai}. They also spawned a variation, ``higher Chow groups with modulus" (see \cite{BS}) that rapidly built connections to various subjects of mathematics such as abelianized fundamental groups \cite{Kerz-Saito}, Somekawa $K$-groups and reciprocity functors \cite{Ivorra-Rulling}, \cite{RY}, and motives with modulus \cite{KSY}, to name a few. However, our further attempts to understand the conjectural motivic cohomology for singular schemes through the cycles with modulus were bumping into increasingly complex technical and philosophical issues. Some of these hindrances encouraged the authors to return to the starting point, and to look for and develop some fundamentally new approaches. 

The new approach of this paper may resolve some of the old issues, while it may create a different set of technical problems. For instance, the Milnor range is now represented by higher dimensional cycles, not by $0$-cycles, so that harder algebro-geometric challenges await us. Nevertheless, we choose to work with this new model, because this seems to be leading us further as well as opening new avenues to handle algebraic cycles via some new means and ideas such as deformation theory or non-archimedean analysis, that were thought to be distant from the subject until now.

The particular case studied in depth in this paper is the truncated polynomial ring $k_m:= k[t]/(t^m)$. We show that the Milnor $K$-groups $K_n ^M (k_m)$ can be expressed in terms of our new cycle groups in the Milnor range. The precursors of these theorems  for higher Chow groups were the theorems of \cite{NS} and \cite{Totaro}, and for additive higher Chow groups, the theorems of \cite{BE2} and \cite{R}. Our theorem in this paper is a unification of all those precursors in \emph{ibids.} in a sense. We repeat however that unlike those precursors, our cycle representatives in the Milnor range are now $1$-cycles. In fact, the $0$-cycles do not appear in our groups (see Remark \ref{remk:0-cycle trivial cor}) at all, so our $1$-cycles in the Milnor range form yet the simplest part.

We retain the notations of the cubical version of higher Chow groups (see \S \ref{sec:HC}) for smooth $k$-schemes. For $k_m$, we redefine $\CH^q (k_m,n)$ in \S \ref{sec:cycles mod t^m}, different from the higher Chow groups of \cite{Bloch HC}, but when $m=1$ so that $k_m = k$, we do have the agreement $\CH^q (k_m, n) = \CH^q (k,m)$. In this new theory, for $m \geq 1$ we can easily define the relative group $\CH^q ((k_m, (t)), n)$ (see Definition \ref{defn:rel Chow}). The main theorem, following immediately from Theorem \ref{thm:Milnor}, is:

\begin{thm}\label{thm:Milnor intro}
Let $k$ be a field of characteristic $0$ and let $m, n \geq 1$ be integers. Let $k_m:= k[t]/(t^m)$. Then the graph homomorphism $K_n ^M (k_m) \to \CH^n (k_m, n)$ to the redefined higher Chow group of $k_m$ is an isomorphism. The isomorphism of the relative parts takes the form $(\Omega_{k/\mathbb{Z}} ^{n-1})^{\oplus (m-1)} \simeq \CH^n ((k_m, (t)), n)$, where the former is isomorphic to the relative Milnor $K$-group $K_n ^M (k_m, (t))$.
\end{thm}

We mention a few further new aspects of our  theory. One of them comes from that the ring $k_m$ has several presentations $k[t]/(t^m)$, $\mathcal{O}_{\mathbb{A}^1_k, 0}/(t^m)$ and $k[[t]]/(t^m)$. Usually those who work on algebraic cycles work with the ``algebraic" situations, e.g. cycles over either $k[t]$ or its localization $\mathcal{O}:=\mathcal{O}_{\mathbb{A}^1_k, 0}$. In this paper, we work with cycles over $k[[t]]= \widehat{\mathcal{O}}_{\mathbb{A}^1_k, 0}$, which is henselian. This gives more admissible cycles than the ``algebraic" situation over $\mathcal{O}$, and generally they have a better rationality property. For instance, $y= \sqrt{1+t}$ is of degree $2$ over $\mathcal{O}$, but it is rational over $\widehat{\mathcal{O}}$ because $\sqrt{1+t } = 1 + \frac{t}{2} - \frac{ t^2}{ 8} + \frac{ t^3}{16} - \cdots$. The possibility of using Hensel's lemma could also be a technical benefit. 

On a pair of integral cycles over $k[[t]]$, we put a ``mod $t^m$ equivalence relation" when their pull-backs to $k[[t]]/(t^m)$ are equal (see Definition \ref{defn:mod t^m 2}). This allows us to use intuitions from the deformation theory to study cycles. The other structure that we use in our new model is the non-archimedean $t$-adic metric on $k((t)) = {\rm Frac} (k[[t]])$. This non-archimedean analytic view-point helps us in proving the following Theorem \ref{thm:mod t^m moving intro}, stated in Theorem \ref{thm:mod t^m}. This is a new type moving result up to mod $t^m$ equivalence. This theorem enables us to transport some technical results already known for cycles over $k[t]$ or $\mathcal{O}_{\mathbb{A}^1_k, 0}$, to our cycles over $k[[t]]$:

\begin{thm}\label{thm:mod t^m moving intro}
Let $k$ be a field. Let $\mathcal{O}:= \mathcal{O}_{\mathbb{A}^1_k, 0}$ and let $\widehat{\mathcal{O}}:= \widehat{\mathcal{O}}_{\mathbb{A}^1_k, 0}$. For the completion homomorphism $\xi^n: z^n_{\mathfrak{m}} (\mathcal{O}, n)^c \to z^n_{\widehat{\mathfrak{m}}} (\widehat{\mathcal{O}}, n)^c$, the composition $\xi^n _m : z^n_{\mathfrak{m}} (\mathcal{O}, n)^c \to z^n_{\widehat{\mathfrak{m}}} (\widehat{\mathcal{O}}, n)^c \to z^n (k_m, n):= z^n _{\widehat{\mathfrak{m}}} (\widehat{\mathcal{O}}, n)^c / \sim_{t^m}$ is surjective.
\end{thm}

Here, the superscripts ``$c$" signify that the groups consist of the cycles proper over $\mathcal{O}$ and $\widehat{\mathcal{O}}$, respectively. See Definition \ref{defn:cycle proper}. Although there generally are more cycles over $k[[t]]$, this theorem shows that modulo $t^m$ in the Milnor range, we can still approximate them by those of the algebraic origin. This eventually allows a reduction of the proof of Theorem \ref{thm:Milnor intro} to the graph cycles, providing a great technical simplification.

Since this mod $t^m$ moving lemma of Theorem \ref{thm:mod t^m moving intro} is a new type of result for the studies algebraic cycles, to give some motivations to the reader, let us quickly sketch the idea. The essential point behind the proof of Theorem \ref{thm:mod t^m moving intro} is the notion of coefficient perturbations of Definition \ref{defn:coeff perturb}: when $W \in z^n_{\widehat{\mathfrak{m}}} (\widehat{\mathcal{O}}, n)^c$ is an integral cycle, we show that it is possible to choose a suitable system of equations, for which perturbations of the coefficients may produce good deformations of $W$. The base parameter space is the space of all choices of the coefficients. Some property such as non-emptiness of the solutions is an open condition on this base. For some other properties, we need a flat family. For this, we use a trick that is an explicit version of the flattening stratification theorem. Using so-obtained locally closed nonempty base over which we have a flat family, we prove that we can deform geometrically integral $W$ mod $t^m$ with all the desired properties preserved, such that it comes from the ``algebraic world" over $\mathcal{O}$. In the process, we need to resort to the non-archimedean $t$-adic metric topology. The general integral cycle case is reduced to the geometrically integral case by constructing a Nisnevich cover.

Some follow-up works will treat the cases of off-Milnor range of the relative Chow group of $(k_m, (t))$, with a cycle-theoretic version of the regulator maps on the additive polylogarithmic complex constructed and studied in \cite{Unver} and \cite{Unver2}. \emph{cf.} \cite{Unver regulator}. Its comparison with the regulators in \cite{Park2} and \cite{Park} will also be discussed. Other on-going works deal with cycles over Artin local $k$-algebras with embedding dimensions $\geq 1$.

We remark that our cycle complex seems to have a natural generalization that might be a candidate model for the motivic cohomology of any $k$-scheme with singularities. This also seems to offer a way to define the relative version for any pair $(X, Z)$ of a scheme and its closed subscheme. The verification that this is well-defined and is the correct definition will require nontrivial works. We hope that this paper provides an evidence that either our approach or its variation has a potential to reach the goal of constructing the ultimate motivic cohomology theory for all $k$-schemes.

\bigskip

\textbf{Acknowledgments.} Part of this work was conceived while both of the authors were visiting Professor H\'el\`ene Esnault's workgroup at Freie Universit\"at Berlin. The authors wish to express their deep gratitudes to Professor H\'el\`ene Esnault and Dr. Kay R\"ulling for their kind hospitality. The authors also feel very grateful to Professor Spencer Bloch for his continued encouragements and interest in the project. The authors thank the referees for invaluable comments and suggestions that improved the quality of the article, and for pointing out a few errors on an earlier version of the article. During this research, JP was partially supported by the National Research Foundation of Korea (NRF) grant No. NRF-2018R1A2B6002287
funded by the Korean government (Ministry of Science and ICT), and S\"{U} was supported by the Humboldt Fellowship for Experienced Researchers. 
 
 \bigskip

\textbf{Conventions.} For a scheme $X \to \Spec (R)$ over a discrete valuation ring $R$, we always denote the special fiber by $X_s$, and the generic fiber by $X_{\eta}$.

\section{Recollections, new definitions and basic results}
In this section, we recall and prove some basic definitions and results on higher Chow complexes needed in this paper. A new one over the   truncated polynomial rings $k[t]/(t^m)$ will be defined in \S \ref{sec:cycles mod t^m}, which is the main complex we work with.

\subsection{Recollections of higher Chow cycles}\label{sec:HC}
Let $k$ be a field.
We recall the cubical version of Bloch's higher Chow complexes (\emph{cf.} \cite{Bloch HC}). Let $\mathbb{P}^1_k:=  \Proj (k[u_0, u_1])$, and let $\ov{\square}_k := \mathbb{P}^1_k$, with $y:= u_1/u_0$ as the coordinate. Let $\square_k:= \ov{\square}_k \setminus \{1 \}$. We let $\square_k ^0 = \ov{\square}_k ^0 := \Spec (k)$, and for $n \geq 1$, we let $\square_k ^n$ (resp. $\ov{\square}_k ^n$)  be the $n$-fold product of  $\square_k$ (resp. $\ov{\square}_k$) with itself over $k$. A \emph{face} $F$ of $\square_k ^n$ (resp. $\ov{\square}_k ^n$) is defined to be the closed subscheme given by a finite set of equations of the form $\{ y_{i_1} = \epsilon_1, \cdots, y_{i_u} = \epsilon _u \}$, for an increasing sequence of indices $ 1 \leq i_1 < \cdots < i_u \leq n$, and $\epsilon_j \in \{ 0, \infty \}$. We allow the case of the empty set of equations, i.e. having $F=\square_k ^n$. A \emph{codimension $1$ face} is given by a single such equation. We often write $F_i ^{\epsilon}$ to be the face given by $\{y_i = \epsilon\}$.

For a smooth $k$-scheme $X$, we let $\square^n _{X} :=X \times_k \square_k ^n,$  $\ov{\square}^n _{X}:=X \times_k \ov{\square}_k ^n,$ and define the face $F_{X} $ of  $\square^n _{X}$ to be the pull-back $X\times _{k} F,$ of a face $F$ of $\square^n _k.$  We drop the subscript $X$ from $F_{i,X} ^{\epsilon}$  when no confusion arises. Let $\un{z}^q (X, n)$ be the free abelian group on the set of codimension $q$ integral closed subschemes $Z \subseteq  \square^n _{X}$ that intersect each face properly on $\square^n _{X}.$ For each codimension $1$ face $F_{i,X} ^{\epsilon}$, with $1 \leq i \leq n$ and $\epsilon \in \{ 0, \infty\}$, and an irreducible $Z \in \un{z} ^q (X, n)$, we let $\partial_i ^{\epsilon} (Z)$  be the cycle associated to the scheme-theoretic intersection $Z \cap F_{i,X} ^{\epsilon}.$ By definition, $\partial_i ^{\epsilon} (Z) \in \un{z}^q (X, n-1)$. This forms a cubical abelian group $(\un{n} \mapsto \un{z}^q (X, n))$, where $\un{n}=\{ 0, 1, \cdots, n \}$, in the sense of \cite[\S 1.1]{Levine SM}. Let $\partial:= \sum_{i=1} ^n (-1)^i (\partial_i ^{\infty} - \partial_i ^0)$ on $\un{z}^q (X, n)$. One checks immediately from the formalism of cubical abelian groups that $\partial \circ \partial  = 0$  and hence one obtains  the associated nondegenerate complex $z^q (X, \bullet):= \un{z}^q (X, \bullet)/ \un{z}^q (X, \bullet)_{\rm degn}$, where $\un{z}^q (X, n)_{\rm degn}$ is the subgroup of  degenerate cycles, i.e. sums of those obtained by pulling back via one of the standard projections $\square_X ^n \to \square_X ^{n-1},$ for $0 \leq i \leq n,$ which omits one of the coordinates on $\square_X ^n.$ This complex $(z^q (X, \bullet), \partial)$ is called the (cubical) \emph{higher Chow complex} of $X$ and its homology groups  $\CH^q (X, n):= {\rm H}_n (z^q (X, \bullet))$ are  the higher Chow groups of $X$. It is a theorem of Voevodsky \cite{Voevodsky} that the higher Chow groups form a universal bigraded ordinary cohomology theory ${\rm H}_{\mathcal{M}} ^{2q-n} (X, \mathbb{Z}(q)):= \CH^q (X, n)$ on the category of smooth $k$-varieties $X$. 

\subsection{Some subgroups}

If we are given an integral closed subscheme $W \subseteq X$, we have a subcomplex $z^q _{W} (X, \bullet) \subseteq z^q (X, \bullet)$ defined as follows: first, let $\un{z}_{W} ^q (X, n) \subseteq \un{z} ^q (X, n)$ be the subgroup generated by integral closed subschemes $Z \subseteq \square_X ^n$ that intersect each $W\times F$ properly on $\square_X ^n$, as well as each $F_{X}=X\times F$, for every face $F$ of $\square^n_{k}$. More precisely, we require that the codimension of $Z \cap (W \times F)$ in $W \times F$ is at least $q$. Modding out by degenerate cycles, we obtain the subcomplex $z^q_{{W}} (X, \bullet) \subseteq z^q (X, \bullet)$. In this paper, we are interested only in the cases when $(X, W)$ is $( \Spec (\widehat{\mathcal{O}}), \widehat{\mathfrak{m}})$ or $ (\Spec (\mathcal{O}), \mathfrak{m})$ where:

\begin{defn}\label{defn:setup}
Let $\mathcal{O}:= \mathcal{O}_{\mathbb{A}_k ^1, 0}$ and $\mathfrak{m}:= \mathfrak{m}_{|mathbb{A}_k^1, 0}$. Let $\widehat{\mathcal{O}}:= \widehat{\mathcal{O}}_{\mathbb{A}_k^1, 0}$ be the completion of $\mathcal{O}$ at $\mathfrak{m}$, and let $\widehat{\mathfrak{m}}$ be its unique maximal ideal. Here $\widehat{\mathcal{O}} \simeq k[[t]]$. For $m \geq 1$, let $k_m:= \widehat{\mathcal{O}}/ (t^m) = k[[t]]/ (t^m)$. We use these notations throughout this paper.
\end{defn}

\begin{remk}\label{remk:0-cycle trivial}

We have $z^{n+1} _{\widehat{\mathfrak{m}}} (\widehat{\mathcal{O}}, n)  = 0$. To see this, suppose that $z^{n+1} _{\widehat{\mathfrak{m}}} (\widehat{\mathcal{O}}, n)  \not = 0$ and let $\mathfrak{p} \in z^{n+1}_{\widehat{\mathfrak{m}}} (\widehat{\mathcal{O}}, n)$ be a closed point on $\square_{\widehat{\mathcal{O}}} ^n$. Here, $[k(\mathfrak{p}):k] < \infty$ so that the image of the composition $\mathfrak{p} \hookrightarrow \square_{\widehat{\mathcal{O}}} ^n \to \Spec (\widehat{\mathcal{O}})$ must be the unique closed point $\widehat{\mathfrak{m}}$ of $\Spec (\widehat{\mathcal{O}})$, i.e. $\mathfrak{p}$ lies in the special fiber of the morphism $\square_{\widehat{\mathcal{O}}} ^n \to \Spec (\widehat{\mathcal{O}})$, contradicting the assumption that $\mathfrak{p} \in z^{n+1}_{\widehat{\mathfrak{m}}} (\widehat{\mathcal{O}}, n)$. Hence $z^{n+1} _{\widehat{\mathfrak{m}}} (\widehat{\mathcal{O}}, n)  = 0$. 
\end{remk}

\begin{remk}\label{remk:0-cycle trivial 2}
We have $z^{n+1} _{\mathfrak{m}} (\mathcal{O}, n) = 0$ as well. The proof is identical to that of Remark \ref{remk:0-cycle trivial} by simply replacing $(\widehat{\mathcal{O}}, \widehat{\mathfrak{m}})$ by $(\mathcal{O}, \mathfrak{m})$. We use Remarks \ref{remk:0-cycle trivial} and \ref{remk:0-cycle trivial 2} later. 
\end{remk}

\begin{cor}\label{cor:no face}
For $n \geq 1$, let $Z \in z^{n}_{\widehat{\mathfrak{m}}} (\widehat{\mathcal{O}}, n)$ be an integral cycle. Then for any proper face $F \subset \square_{\widehat{\mathcal{O}}} ^n$, we have $Z \cap F = \emptyset$. In particular, we have $\partial_i ^{\epsilon} (Z) = 0$ for any $1 \leq i \leq n$ and $i \in \{ 0, \infty\}$, thus $\partial (Z)  = 0$. A similar result holds for $Z \in z^n _{\mathfrak{m}} (\mathcal{O}, n)$. 
\end{cor}

\begin{proof}
If $Z \cap F\not = \emptyset$ for a codimension $1$ face $F \subset \square_{\widehat{\mathcal{O}}} ^n$ given by $\{ y_i = \epsilon\}$, then $\partial_i ^{\epsilon} (W) \not = 0$ in $z^n_{\widehat{\mathfrak{m}}} (\widehat{\mathcal{O}}, n-1)$. But this contradicts Remark \ref{remk:0-cycle trivial} that says $z^n_{\widehat{\mathfrak{m}}} (\widehat{\mathcal{O}}, n-1) = 0$. So $Z$ does not intersect any codimension $1$ face. On the other hand, any proper face is contained in a codimension $1$ face, so the corollary follows.
\end{proof}

\begin{defn}\label{defn:cycle proper}
Let $z^q_{\widehat{\mathfrak{m}}} (\widehat{\mathcal{O}}, n)^c \subset z^q _{\widehat{\mathfrak{m}}} (\widehat{\mathcal{O}}, n)$ be the subgroup generated by the integral cycles $Z \in z^q _{\widehat{\mathfrak{m}}} (\widehat{\mathcal{O}}, n)$ that are proper over $\Spec (\widehat{\mathcal{O}})$. Similarly, define $z^q _{\mathfrak{m}} (\mathcal{O}, n)^c \subset z^q _{\mathfrak{m}} (\mathcal{O}, n)$. 
\end{defn}

There are some technical advantages in working with cycles in $z^q_{\widehat{\mathfrak{m}}} (\widehat{\mathcal{O}}, n)^c $. Firstly we have the following, inspired by the finiteness criterion in \cite[\S 2.3]{KP semi} (Its shortcomings follow soon in Lemma \ref{lem:onlyMilnor}.):

\begin{lem}\label{lem:properness iff}
Let $X$ be a $k$-scheme. Let $W \subset \square_{X} ^n$ be a nonempty closed subscheme and let $\ov{W} \subset \ov{\square}_{X} ^n$ be its Zariski closure. Then $W \to X $ is proper if and only if $W = \ov{W}$.
\end{lem}

\begin{proof}
$(\Rightarrow)$ The structure morphism $W \to X$ factors into the composite $W \hookrightarrow \ov{\square}_{X} ^n \to X$. Since the composite is assumed to be proper and the second morphism is separated by \cite[Theorem II-4.9, p.103]{Hartshorne}, the inclusion $W \hookrightarrow \ov{\square}_{X} ^n$ is proper by \cite[Corollary II-4.8(e), p.102]{Hartshorne}. In particular $W$ is closed in $\ov{\square}_{X} ^n$ and $W= \ov{W}$.

$(\Leftarrow)$ If $W= \ov{W}$, then $W \hookrightarrow \ov{\square}_{X} ^n$ is closed, in particular it is a proper morphism by \cite[Corollary II-4.9(a), p.102]{Hartshorne}. Hence composed with the proper projective 
morphism $\ov{\square}_{X} ^n \to X$, the composite $W \to X$ is proper by \cite[Corollary II-4.8(b), p.102]{Hartshorne}.
\end{proof}

\begin{lem}\label{lem:cycle proper}
Let $W \in z^q_{\widehat{\mathfrak{m}}} (\widehat{\mathcal{O}}, n)^c $ be a nonempty integral cycle. Then we have:
\begin{enumerate}
\item $W$ is closed in $\ov{\square}_{\widehat{\mathcal{O}}} ^n$, so that its Zariski closure $\ov{W}$ is $W$ itself. 
\item The structure morphism $W \to \Spec (\widehat{\mathcal{O}})$ is surjective.
\end{enumerate}
A similar assertion holds for $W \in z^q _{\mathfrak{m}} (\mathcal{O}, n)^c$.
\end{lem}

\begin{proof}
(1) is a corollary to Lemma \ref{lem:properness iff}.

If $W \to {\rm Spec} (\widehat{\mathcal{O}})$ is not dominant, then $W = W_s$, which violates the assumption that $W \in z^q _{\widehat{\mathfrak{m}}} (\widehat{\mathcal{O}}, n)$. Hence $W \to \Spec (\widehat{\mathcal{O}})$ is dominant. Now, being proper and dominant, it must be surjective, proving (2).
\end{proof}

The above Lemmas \ref{lem:properness iff} and \ref{lem:cycle proper} give interesting results for nonempty integral cycles in $z^q_{\widehat{\mathfrak{m}}} (\widehat{\mathcal{O}}, n)^c$. However, this non-emptiness assumption is nontrivial:

\begin{lem}\label{lem:onlyMilnor}
The group $z^q_{\widehat{\mathfrak{m}}} (\widehat{\mathcal{O}}, n)^c$ is nonzero if and only if $q=n$.
\end{lem}

\begin{proof}
If $q \geq n+2$, the group $z^q_{\widehat{\mathfrak{m}}} (\widehat{\mathcal{O}}, n)^c = 0$ since $\dim (\square_{\widehat{\mathcal{O}}} ^n) = n+1$. If $q=n+1$, it is zero by Remark \ref{remk:0-cycle trivial}. Let $q<n$ and suppose that $z^q_{\widehat{\mathfrak{m}}} (\widehat{\mathcal{O}}, n)^c  \not = 0$, so there is a nonempty integral closed subscheme $W \subset \square_{\widehat{\mathcal{O}}} ^n$ that is proper over $\Spec (\widehat{\mathcal{O}})$. Then $W \to \Spec (\widehat{\mathcal{O}})$ is a proper morphism affine schemes, so that it must be finite. (See \cite[Exercise II-4.6, p.106]{Hartshorne}) This forces $q =n$, a contradiction. On the other hand, when $q=n$, certainly $z^q_{\widehat{\mathfrak{m}}} (\widehat{\mathcal{O}}, n)^c $ is nonzero, for instance, they contain all the graph cycles given by $\{y_1 = c_1, \cdots, y_n = c_n \}$ for $c_i \in \widehat{\mathcal{O}} ^{\times}$.
\end{proof}

This apparent shortcoming shows that assuming properness over $\Spec (\widehat{\mathcal{O}})$ for all codimension is too restrictive except for the Milnor range. We thus consider the following a bit weaker cycles with ``partial compactness" defined inductively:

\begin{defn}\label{defn:cycle pproper}
For integers $q \leq n$, define the subgroup $z^q _{\widehat{\mathfrak{m}}} (\widehat{\mathcal{O}}, n)^{pc} \subset z^q_{\widehat{\mathfrak{m}}} (\widehat{\mathcal{O}}, n) $ as follows: 
\begin{enumerate}
\item If $n  \leq q$, we let $z^q _{\widehat{\mathfrak{m}}} (\widehat{\mathcal{O}}, n)^{pc} := z^q _{\widehat{\mathfrak{m}}} (\widehat{\mathcal{O}}, n)^{c}$.
\item Suppose $n > q$. Inductively, suppose $z^q _{\widehat{\mathfrak{m}}} (\widehat{\mathcal{O}}, i)^{pc}$ is defined for each $0 \leq i \leq n-1$. Then let $z^q_{\widehat{\mathfrak{m}}} (\widehat{\mathcal{O}}, n)^{pc}$ be the subgroup of cycles $Z \in z^q_{\widehat{\mathfrak{m}}} (\widehat{\mathcal{O}}, n)$ such that $\partial (Z) \in z^q _{\widehat{\mathfrak{m}}} (\widehat{\mathcal{O}}, n-1)^{pc}$.
\end{enumerate}
By construction, $\partial$ maps $z^q_{\widehat{\mathfrak{m}}} (\widehat{\mathcal{O}}, n)^{pc}$ into $z^q_{\widehat{\mathfrak{m}}} (\widehat{\mathcal{O}}, n-1)^{pc}$, and we have $\partial \circ \partial = 0$ so that $z^q_{\widehat{\mathfrak{m}}} (\widehat{\mathcal{O}}, \bullet)^{pc}$ is a subcomplex of $z^q_{\widehat{\mathfrak{m}}} (\widehat{\mathcal{O}}, \bullet )$. 
We can similarly define $z^q _{\mathfrak{m}} (\mathcal{O}, n)^{pc}$. Define $\CH^q_{\widehat{\mathfrak{m}}} (\widehat{\mathcal{O}}, n)^{pc}:= {\rm H}_n (z^q_{\widehat{\mathfrak{m}}} (\widehat{\mathcal{O}}, \bullet )^{pc} )$ and similarly define $\CH^q_{\mathfrak{m}} (\mathcal{O}, n)^{pc}$.
\end{defn}

\begin{remk}
Our definition does not necessarily imply that for each $1 \leq i \leq n$ and $\epsilon \in \{ 0, \infty \}$, the individual face operator $\partial_i ^{\epsilon}$ maps $z^q_{\widehat{\mathfrak{m}}} (\widehat{\mathcal{O}}, n)^{pc}$ into $z^q_{\widehat{\mathfrak{m}}} (\widehat{\mathcal{O}}, n-1)^{pc}$, unlike the boundary operator $\partial$.
\end{remk}

One good side of our definition is the following:

\begin{cor}\label{cor:pc inj}
We have 
$$\partial (z^n_{\mathfrak{m}} (\mathcal{O}, n+1)^{pc}) = \partial (z^n_{\mathfrak{m}} (\mathcal{O}, n+1)) \cap z^n_{\mathfrak{m}} (\mathcal{O}, n)^c, 
\partial (z^n_{\widehat{\mathfrak{m}}} (\widehat{\mathcal{O}}, n+1)^{pc}) = \partial (z^n_{\widehat{\mathfrak{m}}} (\widehat{\mathcal{O}}, n+1)) \cap z^n_{\widehat{\mathfrak{m}}} (\widehat{\mathcal{O}}, n)^c.$$
So, the maps $\CH^n_{\mathfrak{m}} (\mathcal{O}, n)^{pc} \to \CH^n _{\mathfrak{m}} (\mathcal{O}, n)$ and $\CH^n_{\widehat{\mathfrak{m}}} (\widehat{\mathcal{O}}, n)^{pc} \to \CH^n_{\widehat{\mathfrak{m}}} (\widehat{\mathcal{O}}, n)$ are injections.
\end{cor}

We thank the referee for suggesting Corollaries \ref{cor:pc nb} and \ref{cor:pc nb2}:

\begin{cor}\label{cor:pc nb}
For $n=q+1$, we have 
\begin{equation}\label{eqn:pc nb-1}
\tuborg \ker (\partial: z^q _{\mathfrak{m}} (\mathcal{O}, q+1) \to z^q _{\mathfrak{m}} (\mathcal{O}, q)) \subseteq z^q _{\mathfrak{m}} (\mathcal{O}, q+1) ^{pc},\\
\ker (\partial: z^q _{\widehat{\mathfrak{m}}} (\widehat{\mathcal{O}}, q+1) \to z^q _{\widehat{\mathfrak{m}}} (\widehat{\mathcal{O}}, q)) \subseteq z^q _{\widehat{\mathfrak{m}}} (\widehat{\mathcal{O}}, q+1) ^{pc}.\sluttuborg
\end{equation}
For $n \geq q+2$, we have 
\begin{equation}\label{eqn:pc nb-2}
z^q _{\mathfrak{m}} (\mathcal{O}, n) ^{pc} = z^q _{\mathfrak{m}} (\mathcal{O}, n) \mbox{ and } z^q _{\widehat{\mathfrak{m}}} (\widehat{\mathcal{O}}, n) ^{pc} = z^q _{\widehat{\mathfrak{m}}} (\widehat{\mathcal{O}}, n).
\end{equation}
\end{cor}

\begin{proof}
We do it just for $\mathcal{O}$; the other case is similar. Note that the trivial subgroup of $z^q_{\mathfrak{m}} (\mathcal{O}, q)$ is a subgroup of $ z^q_{\mathfrak{m}} (\mathcal{O}, q)^{pc}$ as well so that the boundaries of the members of $ \ker (\partial: z^q _{\mathfrak{m}} (\mathcal{O}, q+1) \to z^q _{\mathfrak{m}} (\mathcal{O}, q))$ belong to $ z^q_{\mathfrak{m}} (\mathcal{O}, q)^{pc}$. This means $\ker (\partial: z^q _{\mathfrak{m}} (\mathcal{O}, q+1) \to z^q _{\mathfrak{m}} (\mathcal{O}, q)) \subseteq z^q_{\mathfrak{m}} (\mathcal{O}, q+1)^{pc}$, proving the first assertion. 

Since $\partial \circ \partial = 0$, we immediately have $z^q_{\mathfrak{m}} (\mathcal{O}, n) \subseteq z^q_{\mathfrak{m}} (\mathcal{O}, n) ^{pc}$ for $n \geq q+2$. This proves the second assertion.
\end{proof}

\begin{cor}\label{cor:pc nb2}
For $n \not = q$, we have the equalities
$$\CH^q _{\mathfrak{m}} (\mathcal{O}, n) ^{pc} = \CH^q _{\mathfrak{m}} (\mathcal{O}, n) \mbox{ and } \CH^q _{\widehat{\mathfrak{m}}} (\widehat{\mathcal{O}}, n) ^{pc} = \CH^q _{\widehat{\mathfrak{m}}} (\widehat{\mathcal{O}}, n).$$
\end{cor}

\begin{proof}
We do it just for $\mathcal{O}$; the other case is similar. When $n \leq q-2$, we have $z^q_{\mathfrak{m}} (\mathcal{O}, n)^{pc} = z^q _{\mathfrak{m}} (\mathcal{O}, n) = 0$, because $\dim (\square^n _{\mathcal{O}}) = n+1$. When $n= q-1$, we have $z^q_{\mathfrak{m}} (\mathcal{O}, n)^{pc} = z^q _{\mathfrak{m}} (\mathcal{O}, n) = 0$ by Remark \ref{remk:0-cycle trivial 2}. (See Remark \ref{remk:0-cycle trivial} for the case of $\widehat{O}$) Hence $\CH^q_{\mathfrak{m}} (\mathcal{O}, n)^{pc} = \CH^q_{\mathfrak{m}} (\mathcal{O}, n)=0$ for $n < q$. When $n \geq q+2$, by \eqref{eqn:pc nb-2} of Corollary \ref{cor:pc nb}, the equality $\CH^q_{\mathfrak{m}} (\mathcal{O}, n)^{pc} = \CH^q_{\mathfrak{m}} (\mathcal{O}, n)$ holds. When $n= q+1$, the injective map $\CH^q_{\mathfrak{m}} (\mathcal{O}, n)^{pc} \to \CH^q_{\mathfrak{m}} (\mathcal{O}, n)$ of Corollary \ref{cor:pc inj} reads
\begin{equation}\label{eqn:pc nb2-2}
\frac{\ker (\partial: z^q _{\mathfrak{m}} (\mathcal{O}, q+1)^{pc} \to z^q _{\mathfrak{m}} (\mathcal{O}, q)^{pc})}{\partial z^q _{\mathfrak{m}} (\mathcal{O}, q+2)^{pc}} \to \frac{\ker (\partial: z^q _{\mathfrak{m}} (\mathcal{O}, q+1) \to z^q _{\mathfrak{m}} (\mathcal{O}, q))}{\partial z^q _{\mathfrak{m}} (\mathcal{O}, q+2)}.
\end{equation}
By  \eqref{eqn:pc nb-1} and \eqref{eqn:pc nb-2} of Corollary \ref{cor:pc nb}, both the numerators and the denominators of \eqref{eqn:pc nb2-2} are equal. Hence $\CH^q_{\mathfrak{m}} (\mathcal{O}, n)^{pc} \to \CH^q_{\mathfrak{m}} (\mathcal{O}, n)$ is the identity map.
\end{proof}

\begin{remk}\label{remk:n=q}
We guess that the equalities of Corollary \ref{cor:pc nb2} extend to the case when $n=q$ as well, but we have only partial results in this direction. 

When $n=q=1$, we have $z^1_{\mathfrak{m}} (\mathcal{O}, 1) ^{pc} = z^1 _{\mathfrak{m}} (\mathcal{O}, 1)$ and $z^1_{\widehat{\mathfrak{m}}} (\widehat{\mathcal{O}}, 1) ^{pc} =z^1_{\widehat{\mathfrak{m}}} (\widehat{\mathcal{O}}, 1)$. We prove it for $\mathcal{O}$; the case of $\widehat{\mathcal{O}}$ is similar.

An integral cycle $Z \in z^1_{\mathfrak{m}} (\mathcal{O}, 1)$ is given by an irreducible polynomial $f\in \mathcal{O}[y_1]$. Since its intersection with the faces $\{ y_ 1 = 0\}$ and $\{ y_1 =\infty\}$ are both empty, both the leading coefficient $c$ and the constant term are units in $\mathcal{O}^{\times}$. Replacing $f$ by $c^{-1} f$, we see that $f$ is a monic irreducible polynomial. In particular $Z \to \Spec (\mathcal{O})$ is finite, hence proper, so that $Z \in z^1_{\mathfrak{m}} (\mathcal{O}, 1)^{pc}$. Hence we have the equality. 

By the definition of the groups with the superscripts ``pc", this implies that $z^1_{\mathfrak{m}} (\mathcal{O}, 2)^{pc} = z^1_{\mathfrak{m}} (\mathcal{O}, 2)$ and $z^1_{\widehat{\mathfrak{m}}} (\widehat{\mathcal{O}}, 2)^{pc} = z^1_{\widehat{\mathfrak{m}}} (\widehat{\mathcal{O}}, 2)$. Hence we deduce that $\CH^1_{\mathfrak{m}} (\mathcal{O}, 2)^{pc} = \CH^1_{\mathfrak{m}} (\mathcal{O}, 2)$ and $\CH^1_{\widehat{\mathfrak{m}}} (\widehat{\mathcal{O}}, 2)^{pc} = \CH^1_{\widehat{\mathfrak{m}}} (\widehat{\mathcal{O}}, 2)$. 

When $n=q \geq 2$, for $\mathcal{O}$ we have a partial result in Lemma \ref{lem:grOsurj} that when $k$ is infinite the equality holds. However we do not know much about it for $\widehat{\mathcal{O}}$. See Remark \ref{remk:tech diff 1} as well.
\end{remk}

\subsection{Cycles modulo $t^m$}\label{sec:cycles mod t^m}

\begin{defn}\label{defn:mod t^m}
Let $m \geq 1$ be an integer. Let $X$ be an integral ${\rm Spec}(\widehat{\mathcal{O}})$-scheme and let $Z_1, Z_2 \subset X$ be two \emph{integral} closed subschemes of $X$. We allow the case when $Z_1$ or $Z_2$ is the empty scheme. We say that \emph{$Z_1$ and $Z_2$ are equivalent mod $t^m$}, if we have the equality $Z_1 \times_{\Spec (\widehat{\mathcal{O}})} \Spec (\widehat{\mathcal{O}}/ (t^m)) = Z_2 \times_{\Spec (\widehat{\mathcal{O}})} \Spec (\widehat{\mathcal{O}}/ (t^m))$ as closed subschemes of $X \times_{\Spec (\widehat{\mathcal{O}})} \Spec (\widehat{\mathcal{O}}/ (t^m))$. 

We can extend this notion to algebraic cycles on $X$ by extending $\mathbb{Z}$-linearly. 
\end{defn}

\begin{remk}
It might be tempting to define the mod $t^m$-equivalence on each pair of closed subschemes $Z_1$ and $Z_2$ as long as we have $Z_1 \times_{\Spec (\widehat{\mathcal{O}})} \Spec (\widehat{\mathcal{O}}/ (t^m)) = Z_2 \times_{\Spec (\widehat{\mathcal{O}})} \Spec (\widehat{\mathcal{O}}/ (t^m))$. But this finer relation may result in some technically very undesirable effects in dealing with algebraic cycles. One of such problems is that this ``tempting" definition often identifies an irreducible closed subscheme with possibly reducible ones, and this makes an analysis of the behaviors of algebraic cycles very intractable. We thus put this mod $t^m$-equivalence only on pairs of integral closed subschemes with the above equality.
\end{remk}

\begin{defn}\label{defn:mod t^m 2}
For two integral schemes $Z_1, Z_2 \in z^q_{\widehat{\mathfrak{m}}} (\widehat{\mathcal{O}}, n)$, we say that $Z_1$ and $Z_2$ are \emph{naively mod $t^m$-equivalent}, if their Zariski closures $\ov{Z}_1$, $\ov{Z}_2$ in $\ov{\square}_{\widehat{\mathcal{O}}} ^n$ are mod $t^m$-equivalent in the sense of Definition \ref{defn:mod t^m}. Extend this notion $\mathbb{Z}$-linearly to cycles. We say that $Z_1$ and $Z_2$ are \emph{mod $t^m$-equivalent as higher Chow cycles} and write $Z_1 \sim_{t^m} Z_2$, if the pair $(Z_1, Z_2)$ and all pairs of faces $(Z_1 \cap F , Z_2 \cap F)$ for each face $F \subset \square_{\widehat{\mathcal{O}}} ^n$ are all naively mod $t^m$-equivalent. 

For simplicity, when $Z_1, Z_2$ are mod $t^m$-equivalent as higher Chow cycles, we will simply say they are mod $t^m$-equivalent. 
 \end{defn}

The inductive nature of the definition of mod $t^m$-equivalence shows:

\begin{lem}\label{lem:boundary op}
The boundary operator $\partial$ of the complex $z^q_{\widehat{\mathfrak{m}}} (\widehat{\mathcal{O}}, \bullet)$ induces the boundary operator, also denoted by $\partial$, on the mod $t^m$ groups $z^q_{\widehat{\mathfrak{m}}} (\widehat{\mathcal{O}}, n)/\sim_{t^m}$, turning them into a complex. Similarly, we obtain the mod $t^m$ complex $z^q_{\widehat{\mathfrak{m}}} (\widehat{\mathcal{O}}, \bullet)^{pc}/ \sim _{t^m}$.
\end{lem}

To avoid a technical difficulty (see Remark \ref{remk:tech diff 1}), we will use $z^q_{\widehat{\mathfrak{m}}} (\widehat{\mathcal{O}}, \bullet)^{pc}/ \sim _{t^m}$:

\begin{defn}\label{defn:mod t^m cx}
Let $m \geq 1$, $q, n \geq 0$ be integers. Define
\begin{equation}\label{eqn:cycle mod t^m}
z^q (k_m, n):= z^q_{\widehat{\mathfrak{m}}} (\widehat{\mathcal{O}}, n)^{pc}/ \sim_{t^m},
\end{equation}
where $\sim_{t^m}$ is the mod $t^m$-equivalence in Definition \ref{defn:mod t^m 2}. By Lemma \ref{lem:boundary op}, this $z^q (k_m, \bullet)$ is a complex of abelian groups. 
We denote its homology by $\CH^q (k_m, n)$. 
\end{defn}

\begin{remk}\label{remk:0-cycle trivial cor}
The group $z^{n+1} (k_m, n)$ is $0$ because $z^{n+1} _{\widehat{\mathfrak{m}}} (\widehat{\mathcal{O}}, n)  = 0$ by Remark \ref{remk:0-cycle trivial}. Hence the group $z^n (k_m, n)$ is the simplest nontrivial group in our cycle theory.

\end{remk}

\subsection{Relative mod $t^m$ cycle complex}\label{sec:rel cx}
We have $k$-algebra homomorphisms $k \overset{p^{\sharp}}{\to} \widehat{\mathcal{O}} \overset{s^{\sharp}}{\to} k$, where $p^{\sharp}$ is the natural $k$-algebra map and $s^{\sharp}$ is reduction modulo $(t)$. Their composition is the identity of $k$. 
\begin{lem}
The flat structure morphism $p: \Spec (\widehat{\mathcal{O}}) \to \Spec (k)$ given by $p^{\sharp}$ induces the flat pull-back $p^*: z^q (k, \bullet) \to z^q _{\widehat{\mathfrak{m}}} (\widehat{\mathcal{O}}, \bullet)^{pc}$. 
\end{lem}

\begin{proof}
\emph{A priori}, the flat pull-back $p^*$ maps $z^q (k, n)$ into $z^q _{\widehat{\mathfrak{m}}} (\widehat{\mathcal{O}}, n)$. It is enough to show that the image lies in $z^q_{\widehat{\mathfrak{m}}} (\widehat{\mathcal{O}}, n)^{pc}$. If $q>n$, then $z^q (k, n) =0$ so that there is nothing to prove. If $q=n$, then $z^n (k,n)$ and $z^n_{\widehat{\mathfrak{m}}} (\widehat{\mathcal{O}}, n)^{pc} = z^n_{\widehat{\mathfrak{m}}} (\widehat{\mathcal{O}}, n)^{c}$. But every irreducible cycle $Z \in z^n  (k, n)$ is proper over $\Spec (k)$ so that its pull-back $p^* (Z)$ is proper over $\Spec (\widehat{\mathcal{O}})$, thus the assertion holds in this case. 

For $q <n$, the group $z^q_{\widehat{\mathfrak{m}}} (\widehat{\mathcal{O}}, n)^{pc}$ is defined inductively via the boundary operator $\partial$. Suppose the statement of the lemma holds for $n-1$, i.e. $p^*$ maps $z^q (k, n-1)$ into $z^q_{\widehat{\mathfrak{m}}} (\widehat{\mathcal{O}}, n-1)^{pc}$. By \cite[Proposition (1.3)]{Bloch HC}, we have a commutative diagram
$$
\xymatrix{
z^q (k, n) \ar[d] ^{p^*} \ar[r]^{\partial \ \ } & z^q (k, n-1) \ar[d] ^{p^*} \\
z^q_{\widehat{\mathfrak{m}}} (\widehat{\mathcal{O}}, n) \ar[r] ^{\partial \ \ } & z^q_{\widehat{\mathfrak{m}}} (\widehat{\mathcal{O}}, n-1).}
$$
By the induction hypothesis, the right vertical map $p^*$ maps into $z^q_{\widehat{\mathfrak{m}}} (\widehat{\mathcal{O}}, n-1)^{pc}$. Since the above diagram commutes, this means $\partial ( p^* (z^q (k, n))) \subset z^q_{\widehat{\mathfrak{m}}} (\widehat{\mathcal{O}}, n-1)^{pc}$. Hence, by definition $p^* (z^q (k, n)) \subset z^q _{\widehat{\mathfrak{m}}} (\widehat{\mathcal{O}}, n)^{pc}$. Hence by induction, the lemma holds.
\end{proof}

The map $s^{\sharp}$ induces the closed immersion $s: \Spec (k) \to \Spec (\widehat{\mathcal{O}})$. This gives the intersection-restriction to the special fiber $s^*: z^q _{\widehat{\mathfrak{m}}} (\widehat{\mathcal{O}}, {\bullet})^{pc} \to z^q (k, \bullet)$. Since $s^* \circ p^* = {\rm Id}$,  we can regard $z^q (k, \bullet)$ as a subcomplex of $z^q_{\widehat{\mathfrak{m}}} (\widehat{\mathcal{O}}, n)^{pc}$ via $p^*$. Going modulo $t^m$ as in Definition \ref{defn:mod t^m cx}, which does not do anything on $z^q (k,  \bullet)$, we obtain $z^q (k, \bullet) \overset{p^*}{\to} z^q (k_m, \bullet) \overset{s^*}{\to} z^q (k, \bullet)$. Here, we still have $s^* \circ p^* = {\rm Id}$. This gives a splitting \begin{equation}\label{eqn:rel Chow cx}
z^q (k_m, \bullet) = z^q (k, \bullet) \oplus \ker s^*,
\end{equation}
of the complex $z^q (k_m, \bullet)$.
 
\begin{defn}\label{defn:rel Chow}
Define the relative mod $t^m$ cycle complex to be $z^q ((k_m, (t)), \bullet): =\ker s^*$, and its homology is denoted by $\CH^q ((k_m, (t)), n):= {\rm H}_n (z^q ((k_m, (t)), \bullet))$.
\end{defn}

\subsection{Schemes of type $X \otimes_k k_m$ and basic functoriality}\label{sec:X_m}

In \S \ref{sec:cycles mod t^m} and \S \ref{sec:rel cx}, we defined cycle complexes associated to $k_m$. Following the referee's suggestion we can attempt to generalize the construction to schemes of type $X \otimes_k k_m= X \times_{\Spec (k)} \Spec (k_m)$ for any $k$-scheme $X$. For simplicity, when $X$ is a $k$-scheme and $m \geq 1$ is an integer, let $X_m:= X \otimes_k k_m$ and $X^{\wedge}:= X \otimes_k \widehat{\mathcal{O}} = X \times_{\Spec (k)} \Spec (\widehat{\mathcal{O}})$. 

In this generality, it is not yet clear which conditions would give us the ``ultimate correct" definition of the cycle groups, but we can still try to push this direction as far as we can. In the future the situation will get clearer. Here is the provisional definition that generalizes the notions in \S \ref{sec:cycles mod t^m}:

\begin{defn}
Let $X$ be a $k$-scheme. Let $\widehat{\mathfrak{m} }\subset \widehat{\mathcal{O}}$ be the maximal ideal and let $X_{\widehat{\mathfrak{m}}} := X \times_k \{\widehat{\mathfrak{m}}\} \subset X^{\wedge}$ be the closed subscheme. Let $z^q_{\{X_{\widehat{\mathfrak{m}}}\}} (X^{\wedge}, n)^c \subseteq z^q_{\{X_{\widehat{\mathfrak{m}}}\}} (X^{\wedge}, n)$ be the subgroup generated by integral closed subschemes $W \subset X^{\wedge} \times_k \square_k ^n$ that are proper over $X^{\wedge}$. Since the morphism $W \to X^{\wedge}$ is proper and affine, it must be finite so that only when $q \geq n$ we may have a possibly nontrivial group.

We define $z^q_{\{X_{\widehat{\mathfrak{m}}}\}} (X^{\wedge}, n)^{pc} \subseteq z^q_{\{X_{\widehat{\mathfrak{m}}}\}} (X^{\wedge}, n)$ inductively, by imitating what we did before. Namely, for $q \geq n$, we define $z^q_{\{X_{\widehat{\mathfrak{m}}}\}} (X^{\wedge}, n)^{pc}:= z^q_{\{X_{\widehat{\mathfrak{m}}}\}} (X^{\wedge}, n)^{c}$. Suppose that $z^q_{\{X_{\widehat{\mathfrak{m}}}\}} (X^{\wedge}, i)^{pc}$ is defined for all $ i \leq n-1$. Then let $z^q_{\{X_{\widehat{\mathfrak{m}}}\}} (X^{\wedge}, n)^{pc}$ be the subgroup of cycles $Z \in z^q_{\{X_{\widehat{\mathfrak{m}}}\}} (X^{\wedge}, n)$ such that $\partial (Z) \in z^q_{\{X_{\widehat{\mathfrak{m}}}\}} (X^{\wedge}, n-1)^{pc}$. By definition, this gives a complex with respect to the boundary operator $\partial$. 

As in Definition \ref{defn:mod t^m 2}, for $m \geq 1$, we define the mod $t^m$-equivalence inductively on integral cycles in $z^q_{\{X_{\widehat{\mathfrak{m}}}\}} (X^{\wedge}, n)^{pc}$ as well using $- \times_{X^{\wedge}} X_m = - \otimes_{\widehat{\mathcal{O}}} k_m$. Define $z^q (X_m, n):= z^q_{\{X_{\widehat{\mathfrak{m}}}\}} (X^{\wedge}, n)^{pc}/ \sim _{t^m}$. One checks immediately that this gives a complex $z^q (X_m, \bullet)$ as in Lemma \ref{lem:boundary op}. We define $\CH^q (X_m, n):= {\rm H}_n (z^q (X_m, \bullet))$. 
\end{defn}

For the rest of \S \ref{sec:X_m}, we discuss some basic functoriality properties, namely the existence of the push-forward for a proper morphism and the pull-back for a flat morphism. The groups $\CH^q (X_m,n)$ have two types of relations: the first is given by the boundaries of cycles from $z^q_{\{X_{\widehat{\mathfrak{m}}}\}} (X^{\wedge}, n+1)^{pc}$ and the second is given by the mod $t^m$-equivalence. So, to discuss some basic functoriality properties, we need to show that the pull-backs and push-forwards respect both those relations. Fortunately, for the usual push-forwards and pull-backs in the sense of \cite[\S 1.4, \S 1.7]{Fulton}, it is already known by \cite[Proposition (1.3)]{Bloch HC} that they respect the first type of relations given by the boundaries, if we ignore the superscripts $pc$. 

\begin{prop}\label{prop:ppf}
Let $f: X \to Y$ be a proper morphism of $k$-schemes and let $f: X^{\wedge} \to Y^{\wedge}$ also denote the induced proper morphism. Then the push-forward $f_* : z^q_{\{ X_{\widehat{\mathfrak{m}}} \}} (X^{\wedge}, n) \to z^q_{\{ Y_{\widehat{\mathfrak{m}}} \}} (Y^{\wedge}, n)$ satisfies the following properties: 
\begin{enumerate}
\item $f_*$ sends $z^q_{\{ X_{\widehat{\mathfrak{m}}} \}} (X^{\wedge}, n)^{pc}$ into $z^q_{\{ Y_{\widehat{\mathfrak{m}}} \}} (Y^{\wedge}, n)^{pc}$, and $f_* \partial = \partial f_*$ so that $f_*$ is a morphism of complexes:
$$
\xymatrix{
z^q_{\{ X_{\widehat{\mathfrak{m}}} \}} (X^{\wedge}, n+1)^{pc} \ar[d] ^{f_*} \ar[r] ^{\partial} & z^q_{\{ X_{\widehat{\mathfrak{m}}} \}} (X^{\wedge}, n)^{pc}  \ar[d] ^{f_*} \\
z^q_{\{ Y_{\widehat{\mathfrak{m}}} \}} (Y^{\wedge}, n+1)^{pc} \ar[r] ^{\partial} & z^q_{\{ Y_{\widehat{\mathfrak{m}}} \}} (Y^{\wedge}, n)^{pc}.}
$$
\item $f_*$ respects the mod $t^m$-equivalence, i.e. it induces the right vertical arrow of the following diagram, that makes the diagram commutes:
$$
\xymatrix{
z^q_{\{ X_{\widehat{\mathfrak{m}}} \}} (X^{\wedge}, n)^{pc}  \ar[d] ^{f_*} \ar[r] & z^q_{\{ X_{\widehat{\mathfrak{m}}} \}} (X^{\wedge}, n)^{pc}  /\sim_{t^m} \ar[d] ^{f_*} \\
z^q_{\{ Y_{\widehat{\mathfrak{m}}} \}} (Y^{\wedge}, n)^{pc}  \ar[r] & z^q_{\{ Y_{\widehat{\mathfrak{m}}} \}} (Y^{\wedge}, n)^{pc}   /\sim_{t^m}. 
}
$$
\end{enumerate}
\end{prop}

\begin{proof}
Without the superscripts $pc$, we already know by \cite[Proposition (1.3)]{Bloch HC} (which uses \cite[Theorem 6.2(a), p.98]{Fulton}) that the push-forwards $f_*$ are compatible with the boundary maps $\partial$. For (1), we need to check that this still holds after putting the superscripts $pc$. 

For $q \geq n$, an integral cycle $Z \in z^q_{\{ X_{\widehat{\mathfrak{m}}} \}} (X^{\wedge}, n)^{pc} $ is proper over $X^{\wedge}$ by definition. Hence its image under the proper morphism $f: X^{\wedge} \times_k \square_k ^n \to Y ^{\wedge} \times_k \square_k ^n$ is again proper over $Y^{\wedge}$. In particular, $f_* (Z) \in z^q_{\{ Y_{\widehat{\mathfrak{m}}} \}} (Y^{\wedge}, n)^{pc}$. 

For $q < n$, since the groups $z^q_{\{ X_{\widehat{\mathfrak{m}}} \}} (X^{\wedge}, n)^{pc}$ and $z^q_{\{ Y_{\widehat{\mathfrak{m}}} \}} (Y^{\wedge}, n)^{pc}$ are defined inductively in such a way that their images under $\partial$ lie in the previous steps $z^q_{\{ X_{\widehat{\mathfrak{m}}} \}} (X^{\wedge}, n-1)^{pc}$ and $z^q_{\{ Y_{\widehat{\mathfrak{m}}} \}} (Y^{\wedge}, n-1)^{pc}$, the known compability of $f_*$ and $\partial$ for the cycle groups without the superscripts $pc$ and the case of $q \geq n$ imply that $f_*$ maps $z^q_{\{ X_{\widehat{\mathfrak{m}}} \}} (X^{\wedge}, n)^{pc}$ into $z^q_{\{ Y_{\widehat{\mathfrak{m}}} \}} (Y^{\wedge}, n)^{pc}$ and $f_* \partial = \partial f_*$ by induction. This proves (1).

The part (2) is an easy application of the following projection formula: \emph{suppose we have a Cartesian diagram of $k$-schemes
$$\xymatrix{  D_P  \ar@{^(->}[r]^{i_P} \ar[d] ^{p|} & P \ar[d] ^p \\
D_Q \ar@{^(->}[r]^{i_Q} & Q,}$$
where $p$ is a proper morphism, $D_P, D_Q$ are effective divisors such that $D_P = p^* (D_Q)$, and $A \subset P$ is a closed subscheme that intersects $D_P$ properly so that $D_P \cdot A = i_P ^* (A)$ is well-defined. Then $p_* (A) \cdot D_Q = p_* (A \cdot p^* ( D_Q)) = p_* (A \cdot D_P)$.}

Its proof is given in \cite[Proposition 2.3-(c), p.34]{Fulton}. The statement in \emph{loc.cit} is given in the cycle group modulo rational equivalence, but in our case we already suppose the proper intersection condition so that the equality of our cycles  holds on the level of cycles.

We apply the above formula: for two integral closed cycles $Z_1, Z_2 \in z^q_{\{ X_{\widehat{\mathfrak{m}}} \}} (X^{\wedge}, n)^{pc} $ such that $Z_1 \sim_{t^m} Z_2$, we take $P = X^{\wedge} \times_k \ov{\square}_k ^n$, $Q= Y^{\wedge} \times_k  \ov{\square}_k ^n$, $A_i = \ov{Z}_{i}$, $D_Q =$ the divisor of $t^m$ in $Q$, and $D_P=$ the divisor of $t^m$ in $P$. The proper map $p= f$. Hence we have $f_* (A_i) \cdot D_Q = f_* ( A_i \cdot f^*( D_Q)) = f_* (A_i \cdot D_P)$. That $Z_1 \sim_{t^m} Z_2$ means $A_1 \cdot D_P = A_2 \cdot D_P$. Hence we deduce that $f_* (A_1) \cdot D_Q = f_* (A_2) \cdot D_Q$, i.e. $f_* (Z_1) \sim_{t^m} f_* (Z_2)$. This proves (2).
\end{proof}

\begin{prop}\label{prop:fpb}
Let $f: X \to Y$ be a flat morphism of $k$-schemes and let $f : X^{\wedge} \to Y^{\wedge}$ also denote the induced flat morphism. Then the pull-back $f^* : z^q_{\{ Y_{\widehat{\mathfrak{m}}} \}} (Y^{\wedge}, n) \to z^q_{\{ X_{\widehat{\mathfrak{m}}} \}} (X^{\wedge}, n)$ satisfies the following properties:
\begin{enumerate}
\item $f^*$ sends $z^q_{\{ Y_{\widehat{\mathfrak{m}}} \}} (Y^{\wedge}, n)^{pc}$ into $z^q_{\{ X_{\widehat{\mathfrak{m}}} \}} (X^{\wedge}, n)^{pc}$, and $f^* \partial = \partial f^*$, so that $f^*$ is a morphism of complexes:
$$
\xymatrix{
z^q_{\{ Y_{\widehat{\mathfrak{m}}} \}} (Y^{\wedge}, n+1)^{pc} \ar[d] ^{f^*} \ar[r] ^{\partial} & z^q_{\{ Y_{\widehat{\mathfrak{m}}} \}} (Y^{\wedge}, n)^{pc}  \ar[d] ^{f^*} \\
z^q_{\{ X_{\widehat{\mathfrak{m}}} \}} (X^{\wedge}, n+1)^{pc} \ar[r] ^{\partial} & z^q_{\{ X_{\widehat{\mathfrak{m}}} \}} (X^{\wedge}, n)^{pc}.}
$$
\item $f^*$ respects the mod $t^m$-equivalence, i.e. it induces the right vertical arrow of the following diagram, that makes the diagram commutes:
$$
\xymatrix{
z^q_{\{ Y_{\widehat{\mathfrak{m}}} \}} (Y^{\wedge}, n)^{pc}  \ar[d] ^{f^*} \ar[r] & z^q_{\{ Y_{\widehat{\mathfrak{m}}} \}} (Y^{\wedge}, n)^{pc}  /\sim_{t^m} \ar[d] ^{f^*} \\
z^q_{\{ X_{\widehat{\mathfrak{m}}} \}} (X^{\wedge}, n)^{pc}  \ar[r] & z^q_{\{ X_{\widehat{\mathfrak{m}}} \}} (X^{\wedge}, n)^{pc}   /\sim_{t^m}. 
}
$$
\end{enumerate}
\end{prop}

\begin{proof}
Without the superscripts $pc$, we already know by \cite[Proposition (1.3)]{Bloch HC} which uses \cite[Proposition 1.7, p.18]{Fulton} that the pull-backs $f^*$ are compatible with the boundary maps $\partial$. For (2), we need to check that this still holds after putting the superscripts $pc$.

For $q \geq n$, an integral cycle $Z \in z^q_{\{ Y_{\widehat{\mathfrak{m}}} \}} (Y^{\wedge}, n)^{pc} $ is proper over $Y^{\wedge}$ by definition. On the other hand, since the pull-back of a proper morphism is again proper, this time over $X^{\wedge}$, we immediately have  that the flat pull-back $f^*(Z) \in z^q_{\{ X_{\widehat{\mathfrak{m}}} \}} (X^{\wedge}, n)^{pc} $.

For $q < n$, since the groups $z^q_{\{ X_{\widehat{\mathfrak{m}}} \}} (X^{\wedge}, n)^{pc}$ and $z^q_{\{ Y_{\widehat{\mathfrak{m}}} \}} (Y^{\wedge}, n)^{pc}$ are defined inductively in such a way that their images under $\partial$ lie in the previous steps $z^q_{\{ X_{\widehat{\mathfrak{m}}} \}} (X^{\wedge}, n-1)^{pc}$ and $z^q_{\{ Y_{\widehat{\mathfrak{m}}} \}} (Y^{\wedge}, n-1)^{pc}$, the known compability of $f^*$ and $\partial$ for the cycle groups without the superscripts $pc$ and the case of $q \geq n$ imply that $f^*$ maps $z^q_{\{ Y_{\widehat{\mathfrak{m}}} \}} (Y^{\wedge}, n)^{pc}$ into $z^q_{\{ X_{\widehat{\mathfrak{m}}} \}} (X^{\wedge}, n)^{pc}$ and $f^* \partial = \partial f^*$ by induction. This proves (1).

The part (2) is an easy application of the following fact: \emph{suppose we have a Cartesian diagram of $k$-schemes
$$\xymatrix{  D_P  \ar@{^(->}[r]^{i_P} \ar[d] ^{f|} & P \ar[d] ^f \\
D_Q \ar@{^(->}[r]^{i_Q} & Q,}$$
where $f$ is a flat morphism, $D_P, D_Q$ are effective divisors such that $D_P = f^* ( D_Q)$, and $A \subset Q$ is a closed subscheme that intersects $D_Q$ properly so that $D_Q \cdot A = i_Q ^* (A)$ is well-defined. Then $f^* (A) \cdot f^* (D_Q) = f^* ( A \cdot D_Q)$.}

Its proof is given in \cite[Proposition 2.3-(d), p.34]{Fulton}. The statement in \emph{loc.cit} is given in the cycle group modulo rational equivalence, but in our case we already suppose the proper intersection condition so that the equality of our cycles  holds on the level of cycles.

We apply the above formula: for two integral closed cycles $Z_1, Z_2 \in z^q_{\{ Y_{\widehat{\mathfrak{m}}} \}} (Y^{\wedge}, n)^{pc} $ such that $Z_1 \sim_{t^m} Z_2$, we take $P = X^{\wedge} \times_k \ov{\square}_k ^n$, $Q= Y^{\wedge} \times_k  \ov{\square}_k ^n$, $A_i = \ov{Z}_{i}$, $D_Q =$ the divisor of $t^m$ in $Q$, and $D_P=$ the divisor of $t^m$ in $P$. Hence we have $f^* (A_i) \cdot D_P = f^* (A_i \cdot D_Q)$. That $Z_1 \sim_{t^m} Z_2$ means $A_1 \cdot D_Q = A_2 \cdot D_Q$. Hence $f^* (A_1) \cdot D_P  = f^* (A_2) \cdot D_P$, i.e. $f^* (Z_1) \sim_{t^m} f^* (Z_2)$. This proves (2).
\end{proof}

\subsection{The non-archimedean norm}\label{sec:nonarch}
 We recall some facts on the non-archimedean $t$-adic metric topology on the local field $k((t))$, needed in \S \ref{sec:perturbation}. Recall that the field $k((t))$ has a natural discrete valuation $v: k((t)) \to \mathbb{Z}$ given by the order of vanishing function $v=\ord_t$ with $v(0):= \infty$. Its ring of integers $\mathcal{O}_{k((t))} = \widehat{\mathcal{O}} =k[[t]]$ is simply $\{ f \in k((t)) \ | \ v(f) \geq 0\}$. We have a norm $|-|_v: k((t)) \to \mathbb{R}$ given by $|f|_v := e^{-v(f)}$. For any integer $M >0$, we have the supremum norm on the vector space $k((t))^M$ given by $|( f_1, \cdots, f_M)|_v:= \sup_{1 \leq i \leq M } |f_i|_v$. This gives the non-archimedean $t$-adic metric topology, which is finer than the Zariski topology on $k((t))^M = \mathbb{A}^M ( k((t)) )$. For any $\alpha_0 \in k((t))^M$, we let ${\mathcal{B}}_N (\alpha_0)$ be the open ball around $\alpha_0$ of radius $e^{-N}$. Here $k[[t]]^M \subset k((t))^M$ is open, while $k[t] ^M \subset k[[t]]^M$ is dense.

\subsection{Milnor $K$-groups}
Let $R$ be a commutative local ring with unity. Recall that the Milnor $K$-ring $K^M_* (R)$ of $R$ is the graded tensor algebra $T_{\mathbb{Z}} (R^{\times})$ of $R^{\times}$ over $\mathbb{Z}$ modulo the two-sided ideal generated by the elements of the form $\{ a \otimes (1-a) \ | \ a, 1-a \in R^{\times} \}$. Its degree $n$ part is the $n$-th Milnor $K$-group $K^M_n (R)$.

 \section{Milnor range I: reciprocity}\label{sec:Milnor}
 The goal of the paper is to prove the following Theorem \ref{thm:Milnor}. In the case  of additive higher Chow groups over fields, similar results were obtained by Bloch-Esnault \cite{BE2} and R\"ulling \cite{R}.
   
 \begin{thm}\label{thm:Milnor}
 Let $k$ be a field of characteristic $0$ and let $m , n \geq 1$ be integers.
 Then we have $\CH^n ((k_m,(t)), n) \simeq  (\Omega_{k/\mathbb{Z}} ^{n-1}) ^{\oplus (m-1)} $. 
 \end{thm}

 The proof of Theorem \ref{thm:Milnor} is largely broken into two parts: the first is to define regulator maps on cycles and  to prove that they vanish on the boundaries, as done in Proposition \ref{prop:main reciprocity} below. The second part, done later in \S \ref{sec:perturbation} and \S \ref{sec:Milnor 2}, is to show that the regulator maps respect the mod $t^m$-equivalence. Here, we emphasize that although we are in the Milnor range,  our representatives are $1$-cycles, unlike the additive Chow group versions  of \cite{BE2} or \cite{R} that used $0$-cycles.  In our discussion, the argument of the first part follows a path similar to one paved in \cite{Park}:

 \begin{prop}\label{prop:main reciprocity}
 Let $k$ be a field of characteristic $0$.
 For each $1 \leq i \leq m-1$, define $\Upsilon_i: z^n _{\widehat{\mathfrak{m}}} (\widehat{\mathcal{O}}, n) \to \Omega_{k/\mathbb{Z}} ^{n-1}$ as follows. Consider the rational form $\gamma_{i,n}:= \frac{ 1}{ t^i} d \log y_1 \wedge \cdots \wedge d \log y_n  \in \Omega^{n}_{\ov{\square}_{\widehat{\mathcal{O}}} ^n/ \mathbb{Z}} ( * \{ t = 0 \})(\log F)$. For each integral $1$-cycle $Z \in z^n _{\widehat{\mathfrak{m}}} (\widehat{\mathcal{O}}, n)$, let $\nu: \widetilde{Z} \to \ov{Z} \hookrightarrow \ov{\square}_{\widehat{\mathcal{O}}} ^n$ be a normalization of the closure $\ov{Z}$ of $Z$ in $\ov{\square}_{\widehat{\mathcal{O}}} ^n$. Define
$\Upsilon_i (Z):= \sum_{ p \in \widetilde{Z}_s} {\rm Tr}_{k(p)/k} \res_p \nu^* \gamma_{i,n} \in \Omega_{k/\mathbb{Z}} ^{n-1},$
and $\mathbb{Z}$-linearly extend $\Upsilon_i$ to all of $z^n_{\widehat{\mathfrak{m}}} (\widehat{\mathcal{O}}, n)$.
Then $\Upsilon_i (\partial W) = 0$ for $W \in z_{\widehat{\mathfrak{m}}} ^n (\widehat{\mathcal{O}} , n+1)$.
 \end{prop}

\begin{proof}

It is enough to prove the statement for any integral $W \in z^n _{\widehat{\mathfrak{m}}} (\widehat{\mathcal{O}}, n+1)$. Let $\ov{W} \subset \ov{\square}_{\widehat{\mathcal{O}}} ^{n+1}$ be the Zariski closure of $W$, which is also integral. For each $1 \leq \ell \leq n+1$ and $\epsilon \in \{ 0, \infty\}$, via the codimension $1$ face map $\iota_{\ell, \epsilon}: \ov{\square}_{\widehat{\mathcal{O}}} ^n \hookrightarrow \ov{\square}_{\widehat{\mathcal{O}}} ^{n+1}$, identify the Zariski closure of $\partial_\ell ^{\epsilon} (W)$ in $\ov{\square}_{\widehat{\mathcal{O}}}^n$ with its image $\ov{\partial_\ell ^{\epsilon} (W)}$ in $\ov{\square}_{\widehat{\mathcal{O}}} ^{n+1}$. Consider the divisor $D:= \sum_{\ell, \epsilon} \{ y_\ell = \epsilon\}$ on $\ov{\square}_{\widehat{\mathcal{O}}} ^{n+1}$. We omit the proof of the following claim, which is easily deduced by a standard argument using a finite sequence of point blow-ups and \cite[Exercise II-7.12, p.171]{Hartshorne}:
 
\textbf{Claim:} \emph{There exists a sequence of blow-ups $\widetilde{\phi}: \widetilde{\ov{\square}_{\widehat{\mathcal{O}}}^{n+1}} \to \ov{\square}_{\widehat{\mathcal{O}}}^{n+1}$ centered at points, such that for the strict transform $\widetilde{W} \subset \widetilde{\ov{\square}_{\widehat{\mathcal{O}}}^{n+1}}$ of $\ov{W}$ 
and the restriction $\phi:= \widetilde{\phi}|_{\widetilde{W}}: \widetilde{W} \to \ov{W} \hookrightarrow \ov{\square}_{\widehat{\mathcal{O}}} ^{n+1}$, we have the following properties: $(1)$ each irreducible component of the strict transform $\phi^! ( \ov{\partial_\ell ^{\epsilon} (W)})$ of the $1$-cycle $\ov{\partial_\ell ^{\epsilon} (W)}$ is regular. We let $\phi^{!} (D):= \sum_{\ell, \epsilon} \phi^! ( \ov{\partial_\ell ^{\epsilon} (W)})$, the strict transform of $\ov{D \cap W}$; \  $(2)$ each closed point $p \in \widetilde{W}_s=\phi^* \{ t = 0 \}$ satisfies exactly one of the following three possibilities:
\begin{enumerate}
\item [(2-i)] $p$ belongs to a unique irreducible component of $\widetilde{W}_s$, but does not meet $\phi^! (D)$.
\item [(2-ii)] $p$ belongs to a unique irreducible component of $\widetilde{W}_s$, and belongs to precisely one irreducible component of $\phi^! (D)$.
\item [(2-iii)] $p$ belongs to exactly two irreducible components of $\widetilde{W}_s$, but does not meet $\phi^! (D)$.
\end{enumerate}
}

\bigskip

Going back to the proof of the proposition, notice that the irreducible components of $\phi^! ( \ov{\partial_\ell ^{\epsilon} (W)})$ are all regular, and are in one-to-one correspondence with the irreducible components of $\ov{\partial_\ell ^{\epsilon} (W)}$ via $\phi$. Hence each component of $\phi^! ( \ov{\partial_\ell ^{\epsilon} (W)})$ gives a normalization of the Zariski closure in $\ov{\square}_{\widehat{\mathcal{O}}} ^n$ of the corresponding component of $\partial_\ell ^{\epsilon }( W)$. 
Express the special fiber $\widetilde{W}_s$ as the union of (not necessarily regular) irreducible projective curves $C_1, \cdots, C_M$.

 We use the theory of Parshin-Lomadze residues associated to pseudo-coefficient fields (see \cite[Definitions 4.1.1, 4.1.3]{Yekutieli}). For each generic point of  $C_j$ seen as a point of the scheme $\widetilde{W}$, choose a pseudo-coefficient field $\sigma_j$. Consider the Parshin-Lomadze residue $\Xi_{\sigma_j}:= \res_{(\widetilde{W}, C_j), \sigma_j} \phi^* (\gamma_{i, n+1})$ along the chain $(\widetilde{W}, C_j)$ for the choice of $\sigma_j$. For each $1 \leq j \leq M$, this $\Xi_{\sigma_j}$ is a rational absolute K\"ahler $n$-form on $C_j$.

Let $p \in \widetilde{W}_s$. By our construction of $\widetilde{W}$ in the above claim, for the point $p$, exactly one of (2-i), (2-ii), and (2-iii) holds. 

If (2-i) holds for $p$, then let $C_j$ be the unique component of $\widetilde{W}_s$ with $p \in C_j$. Since $p$ does not lie over any face $\ov{\partial_{\ell} ^{\epsilon} (W)}$ for $1 \leq \ell \leq n+1$, $\epsilon\in \{0, \infty \}$, from the shape of $\gamma_{i, n+1}$, the form $\Xi_{\sigma_j}  = \res_{(\widetilde{W}, C_j), \sigma_j} (\phi^* \gamma_{i, n+1})$ is regular at $p$ so that we have $\res_{p \in C_j} (\Xi_{\sigma_j}) = 0$.

If (2-iii) holds for $p$, then let $C_j, C_{j'}$ with $j \not = j'$ be the two distinct components of $\widetilde{W}_s$ such that $p \in C_j \cap C_{j'}$. Here, again $p$ does not lie over any face $\ov{\partial_{\ell} ^{\epsilon} (W)}$  for $1 \leq \ell \leq n+1$, $\epsilon\in \{0, \infty \}$, therefore by 
\cite[Theorem 4.2.15-(a)]{Yekutieli}, we have  $\res_{p \in C_j} (\Xi _{\sigma_j}) +\res_{p \in C_{j'}} (\Xi_{\sigma_{j'}}) = 0$. 

Now suppose (2-ii) holds for $p$. Thus there exist (i) a unique index $1 \leq j(p) \leq M$ with $p \in C_{j(p)}$, (ii) a unique pair $(\ell_0, \epsilon_0)$ with $1 \leq \ell_0 \leq n+1$, $\epsilon _0 \in \{ 0, \infty \}$, and a unique irreducible component $G \subset \phi^! ( \ov{\partial_{\ell_0} ^{\epsilon_0} (W)}) $ such that $p \in G$.  

From the shape of $\gamma_{i, n+1}$, the form $\phi^* \gamma_{i, n+1}$ on $\widetilde{W}$ has a simple (or logarithmic) pole (see \cite[Definition 4.2.10]{Yekutieli}) along $G$, so that the residue of $\phi^* \gamma_{i, n+1}$ along the chain $ (\widetilde{W}, G)$ is independent of the choice of a pseudo-coefficient field for $G$ by \cite[Corollary 4.2.13]{Yekutieli}. On the other hand, by \cite[Theorem 4.2.15-(a)]{Yekutieli}, we have 
\begin{equation}\label{eqn:Milnor pf 0}
\res_{p \in C_{j(p)}} (\res_{(\widetilde{W}, C_{j(p)}), \sigma_{j(p)}} (\phi^* \gamma_{i, n+1})) = - \res_{p \in G} (\res_{(\widetilde{W}, G)} (\phi^* \gamma_{i, n+1})).
\end{equation}

From the shape of $\gamma_{i, n+1}= \frac{1}{t^i} \frac{dy_1}{y_1} \wedge \cdots \wedge \frac{ dy_{n+1}}{y_{n+1}}$ again, since $G \subset \phi^! ( \ov{\partial_{\ell_0} ^{\epsilon_0} (W)})$, we have
\begin{equation}\label{eqn:Milnor pf 1}
\res_{(\widetilde{W}, G)} (\phi^* \gamma_{i, n+1}) = (-1)^{\ell_0} \cdot \iota (G; \ell_0, \epsilon_0) \cdot \sgn (\epsilon_0) \phi^* (\gamma_{i,n+1}^{\ell_0})|_{G},
\end{equation}
where $\iota(G; \ell_0, \epsilon_0)$ is the intersection multiplicity of $G$ in $\partial_{\ell_0} ^{\epsilon_0} (W)$, $\sgn (0):= 1, \sgn (\infty):= -1$, and $\gamma_{i, n+1} ^{\ell_0}:=\frac{1}{t^i} \frac{dy_1}{y_1} \wedge \cdots \widehat{ \frac{ dy_{\ell_0}}{y_{\ell_0}}}\wedge \cdots \wedge \frac{dy_{n+1}}{y_{n+1}} $. 

Now, by the definition of $\Upsilon_i$, we have
$$
(-1)^{\ell_0} \sgn (\epsilon_0) \Upsilon_i (\partial_{\ell_0} ^{\epsilon_0} (W))= (-1)^{\ell_0} \sgn (\epsilon_0)  \sum_{G} \iota (G; \ell_0, \epsilon_0) \sum_{ p \in G_s} \Tr_{k(p)/k} \res_{p \in G} \phi^* (\gamma_{i,n+1} ^{\ell_0})|_G
$$
$$
 =^{\dagger} \sum_G \sum_{p \in G_s} \Tr_{k(p)/k} \res_{p \in G} ( \res_{ (\widetilde{W}, G)} (\phi^* \gamma_{i, n+1}))
 $$
 $$
 =^{\ddagger} - \sum_G \sum_{p \in G_s} \Tr_{k(p)/k} \res_{p \in C_{j(p)}} ( \res_{ (\widetilde{W}, C_{j(p)}), \sigma_{j(p)}} ( \phi^* \gamma_{i, n+1}))
 $$
 \begin{equation}\label{eqn:Milnor pf 2} 
 =^1  -\sum_G \sum_{p \in G_s} \Tr_{k(p)/k} \res_{p \in C_{j(p)}} (\Xi_{\sigma_{j(p)}}),
 \end{equation}
where $\sum_G$ is the sum over all irreducible components of $\phi^! ( \ov{\partial_{\ell_0} ^{\epsilon_0} (W)})$, $\dagger$ holds by \eqref{eqn:Milnor pf 1}, $\ddagger$ holds by \eqref{eqn:Milnor pf 0}, and $=^1$ holds by definition. Note that the set of all points $p \in G_s$ over all irreducible components $G$ of $\phi^! (D)$ is precisely equal to the set $\widetilde{W}_s ^{\rm{ (2-ii)}}$ of all points of $\widetilde{W}_s$ of type (2-ii) in the claim. Hence, taking the sum of \eqref{eqn:Milnor pf 2} over all $1 \leq \ell_0 \leq n+1$ and $\epsilon_0 \in \{ 0, \infty \}$, we obtain {\small 
\begin{equation}\label{eqn:Milnor pf 3}
\Upsilon_i (\partial W)  = - \sum_{\ell = 1} ^{n+1} \sum_{\epsilon \in \{ 0, \infty\}} (-1)^{\ell} \sgn (\epsilon) \Upsilon_i (\partial_{\ell} ^{\epsilon} (W)) = \sum_{p \in \widetilde{W}_s ^{ \rm{ (2-ii)}}} \Tr_{k(p)/k} \res_{p \in C_{j(p)}} (\Xi_{\sigma_{j(p)}}).
\end{equation}
}
On the other hand, for the points of $\widetilde{W}_s$ of type (2-i) and (2-iii), we saw previously that there is no contribution of residues from them. Hence continuing \eqref{eqn:Milnor pf 3}, we have
$$
\Upsilon_i (\partial W) = \sum_{p \in \widetilde{W}_s ^{ \rm{ (2-ii)}}} \Tr_{k(p)/k} \res_{p \in C_{j(p)}} (\Xi_{\sigma_{j(p)}}) = \sum_{j=1} ^M \sum_{p \in C_j} \Tr_{k(p)/k} \res_{p \in C_{j}} (\Xi_{\sigma_{j}})=^{\dagger} 0,
$$
where $\dagger$ holds by the residue theorem (see \cite[Theorem 4.2.15-(b)]{Yekutieli}), i.e. the sum of all residues of a form over a projective curve $\widetilde{W}_s$ is $0$. This shows $\Upsilon_i (\partial W) = 0$ as desired.
 \end{proof}

 The remaining part of the proof of Theorem \ref{thm:Milnor} is to check that the regulator maps in Proposition \ref{prop:main reciprocity} restricted to $z^n _{\widehat{\mathfrak{m}}} (\widehat{\mathcal{O}}, n)^c$ respect the mod $t^m$-equivalence. This requires  further discussions, and the rest of the paper deals with it.

\section{Some perturbation lemmas and the mod $t^m$ moving lemma}\label{sec:perturbation}

 In working with cycles over the complete local ring $\widehat{\mathcal{O}}$,  it is maybe convenient if one can transfer some known results for cycles over $\mathcal{O}$ to cycles over $\widehat{\mathcal{O}}.$ The completion ring homomorphism $\mathcal{O} \to \widehat{\mathcal{O}}$ induces a natural flat pull-back homomorphism $\xi^n : z^q_{\mathfrak{m}} (\mathcal{O} , n) ^? \to z^q _{\widehat{\mathfrak{m}}} (\widehat{\mathcal{O}}, n)^?$, for $?=  pc, \emptyset$, given by $[Z] \mapsto [\widehat{Z}:=\Spec(  \widehat{\mathcal{O}} ) \times_{\Spec (\mathcal{O})} Z]$, but in general, $\xi^n$ is not surjective. The goal of \S \ref{sec:perturbation} is to prove the ``mod $t^m$ moving lemma" in Theorem \ref{thm:mod t^m}, which states that this natural homomorphism induces a surjection modulo $t^m$ in the Milnor range with $?= pc$. In this case $z^n _{\widehat{\mathfrak{m}}} (\widehat{\mathcal{O}}, n)^{c} = z^n _{\widehat{\mathfrak{m}}} (\widehat{\mathcal{O}}, n)^{pc}$.

In this section we suppose $k$ is any field unless specified otherwise. In \S \ref{sec:preperturb}, we discuss some preparatory results needed in what follows. In \S \ref{subsec:perturb}, we discuss various general position results as in Lemmas \ref{lem:small properness}, \ref{lem:small dominant}, \ref{lem:small int}, and \ref{lem:small proper int}, needed in the proof of the mod $t^m$ moving lemma in \S \ref{sec:mod t^m moving}. These results might appear to be related to the Artin approximation theorem \cite{Artin}, but they do not follow from it. The results are stated in terms of schemes over $\widehat{\mathcal{O}} = \widehat{\mathcal{O}}_{\mathbb{A}^1_k, 0}$, but some of them might work for more general integral $k$-schemes with the methods presented here. We leave such generalizations to the reader.

In what follows in \S \ref{sec:perturbation}, to ease the proof, via the automorphism $y \mapsto 1/(1-y)$ of $\mathbb{P}^1$, we identify $(\square, \{ \infty, 0 \})$ with $ ( \mathbb{A}^1, \{ 0, 1 \})$ so that $\square^n \simeq \mathbb{A}^n$, and the faces of $\square^n$ under this identification are given by a finite set of equations of the form $y_{j} = \epsilon_j$ with $\epsilon_j \in \{ 0,  1 \}$.

\subsection{Some preparatory lemmas}\label{sec:preperturb}

We are interested in understanding ``small changes" of a given integral closed subscheme $W \subseteq \square_{\widehat{\mathcal{O}}} ^n$ when we ``perturb" the coefficients of a generating set of the ideal of $W$. So, we introduce:
 
\begin{defn}\label{defn:coeff perturb}
For a closed subscheme $W\subseteq \square_{\widehat{\mathcal{O}}} ^n$, let $\{f_1, \cdots, f_r \} \subset \widehat{\mathcal{O}}[y_1, \cdots, y_n]$ be a set of generators of the ideal of $W$. 

The \emph{coefficient perturbation} of the set $\{f_1, \cdots, f_r \}$ is the set $\{F_1, \cdots, F_r\}$ of polynomials obtained as follows:  for each nonzero monomial term of each of $f_j$ over $1 \leq j \leq r$, consider an indeterminate and a copy of $\mathbb{A}^1_{\widehat{\mathcal{O}}}$. Replace each nonzero coefficient by the corresponding indeterminate. Let $M$ be the total number of them and let $ F_1, \cdots, F_r \in \widehat{\mathcal{O}}[ x_1, \cdots, x_M][y_1, \cdots, y_n]$ be the so-obtained polynomials from $f_1, \cdots, f_r$, respectively. Let $V \subset \mathbb{A}_{\widehat{\mathcal{O}}} ^M \times_{\widehat{\mathcal{O}}} \square_{\widehat{\mathcal{O}}} ^n$ be the closed subscheme defined by the ideal $(F_1, \cdots, F_r)$. We may also say $V$ is the \emph{coefficient perturbation of $W$ with respect to the generators $\{ f_1, \cdots, f_r \}$}. 

For each $\alpha \in \mathbb{A}_{\widehat{\mathcal{O}}} ^M$, we let $V_{\alpha}$ be the fiber over $\alpha$. If $\alpha_0 \in \widehat{\mathcal{O}} ^M$ is the original sequence of coefficients of $\{ f_1, \cdots, f_r\}$, we have $V_{\alpha_0} = W$.
\end{defn}

\begin{exm}
For $n=2$, consider $\{ f_1, f_2 \} = \{ 3 y_1 y_2 ^2 + y_1 + 2y_2 + 1,  - y_1^2 y_2 -5 y_1 + 3\}$. Then the corresponding coefficient perturbation is given by $\{ F_1, F_2 \} = \{ x_1 y_1 y_2^2 + x_2 y_1 + x_3 y_2 + x_4, x_5 y_1 ^2 y_2 + x_6 y_1 + x_7 \}$. 
\end{exm}

The coefficient perturbation depends on the choice of a generating set $\{f_1, \cdots, f_r \}$. If we make a ``bad" choice, then we might end up having undesirable phenomena:
 
\begin{exm}
For $n=2$, consider $W \subseteq \square_{\widehat{\mathcal{O}}} ^2$ defined by $f_1:= y_1 + 1, f_2 := y_2 + 1$. The ideal of $W$ also contains $f_3:= f_1 f_2 = y_1 y_2 + y_1 + y_2 + 1$. 

If we take the coefficient perturbation with respect to just $\{ f_1, f_2 \}$, then we have $F_1 = x_1 y_1 + x_2$, $F_2 = x_3 y_3 + x_4$. So, if $\alpha = (x_1, x_2, x_3, x_4)$ is in the open subset of $\mathbb{A}_{\widehat{\mathcal{O}}} ^4$ given by $x_1\not = 0$ and $x_3 \not = 0$, then $V_{\alpha}$ is given by $(y_1, y_2) = (- x_2/ x_1, - x_4/ x_3)$. In particular, $V_{\alpha} \not = \emptyset$.

However, this time with respect to $\{ f_1, f_2, f_3\}$, with a redundant generator $f_3$, the corresponding coefficient perturbation is given by $F_1 = x_1 y_1 + x_2, F_2= x_3 y_2 + x_4, F_3 = x_5 y_1 y_2 + x_6 y_1 + x_7 y_2 + x_8$. Unfortunately, for $V_{\alpha}$ to be nonempty, we need a necessary condition. Suppose for a choice $\alpha=(x_1, \cdots, x_8)$, we have $V_{\alpha} \not = \emptyset$. Then $F_1 = 0$ and $F_2 = 0$ give $y_1 = - x_2/x_1, y_2 = - x_4 / x_3$ so that by plugging them into $F_3$, we obtain $x_2 x_4 x_5/ (x_1 x_3) - x_2 x_6/ x_1 - x_4 x_7 / x_3 + x_8 = 0$, i.e. we have an algebraic dependence $x_2 x_4 x_5 - x_2 x_3 x_6 - x_1 x_4 x_7 + x_1 x_3 x_8 = 0$ for $x_1, \cdots, x_8$. Hence we can expect to have a nonempty fiber $V_{\alpha}$  only over this proper closed subset of $\mathbb{A}_{\widehat{\mathcal{O}}} ^8$. This is not desirable for our purposes.
\end{exm}

 An aim of \S \ref{sec:preperturb} is to show that when $W \in z^n _{\widehat{\mathfrak{m}}} (\widehat{\mathcal{O}}, n)^c$ is integral of relative dimension $0$ in the Milnor range, it is possible to choose a ``nice" generating set. We will make this precise in what follows.

\begin{lem}\label{lem:gf rel 0}
Let $W \in z^n_{\widehat{\mathfrak{m}}} (\widehat{\mathcal{O}}, n)^c$ be a nonempty integral cycle.  
Then $(1)$ the structure morphism $f: W \to \Spec (\widehat{\mathcal{O}})$ is surjective, flat, and finite, and $(2)$ the generic fiber $W_{\eta}$ is the singleton given by the generic point $\eta_W$ of $W$.
\end{lem}

\begin{proof}
The surjectivity of $f$ was proven in Lemma \ref{lem:cycle proper}. The morphism $f$ is flat by \cite[Proposition (14.5.6), p.217]{EGA4-3} (or \cite[Proposition III-9.7, p.257]{Hartshorne}) because $\Spec (\widehat{\mathcal{O}})$ is a regular scheme of dimension $1$. The morphism $f$ is finite because it is a proper morphism of affine schemes (see \cite[Exercise II-4.6, p.106]{Hartshorne}). This proves (1). 

Since $\dim \ W = 1$, the integral scheme $W$ is the union of the generic point $\eta_W$ of $W$ and its closed points. Here, all the closed points map to the unique closed point of $\Spec (\widehat{\mathcal{O}})$, while $\eta_W$ cannot map to the closed point of $\Spec (\widehat{\mathcal{O}})$ for otherwise $f$ would not be surjective, contradicting Lemma \ref{lem:cycle proper}. Hence $\eta_W$ is the unique point of the generic fiber $W_{\eta}$. This proves (2).
\end{proof}

\begin{prop}\label{prop:ci}
Let $W \in z^n_{\widehat{\mathfrak{m}}} (\widehat{\mathcal{O}}, n)^c$ be a nonempty integral cycle.  
Then $W$ is a complete intersection in $\square_{\widehat{\mathcal{O}}} ^n$ defined by a subset $\{ f_1, \cdots, f_n \} \subseteq \widehat{\mathcal{O}}[y_1, \cdots, y_n]$ of precisely $n$ polynomials of the triangular form
\begin{equation}\label{eqn:triangular O}
\tuborg f_1 (y_1), \\
f_2 (y_1, y_2), \\
\ \ \ \ \ \vdots \\
f_n (y_1, \cdots, y_n), \sluttuborg
\end{equation}
such that $(1)$ $f_i (y_1, \cdots, y_i)$ has $y_i$-degree $\geq 1$ for each $1 \leq i \leq n$, $(2)$ the highest $y_i$-degree term of $f_i$ does not involve any variable other than $y_i$, and $(3)$ the constant term of each $f_i$ is $1$.
\end{prop}

\begin{proof}
This is inspired by \cite[Lemma 2]{Totaro}, but we need some modifications as our base ring is $\widehat{\mathcal{O}}$, not a field. For each $1 \leq i \leq n$, let $W^{(i)} \subset \square_{\widehat{\mathcal{O}}} ^i$ be the image of the projection $\square_{\widehat{\mathcal{O}}} ^n \to \square_{\widehat{\mathcal{O}}} ^i$, $(y_1, \cdots, y_n)\mapsto (y_1, \cdots, y_i)$. Let $W^{(0)}:= \Spec (\widehat{\mathcal{O}})$. Since the map $W= W^{(n)} \to W^{(0)} = \Spec (\widehat{\mathcal{O}})$ is finite and surjective by Lemma \ref{lem:gf rel 0}, we deduce that $W^{(i)} \to W^{(j)}$ is finite and surjective for each pair $0 \leq j < i \leq n$ of indices, and each $W^{(i)} \in z^i_{\widehat{\mathcal{O}}} (\widehat{\mathcal{O}}, i)^c$ is a nonempty integral cycle for each $1 \leq i \leq n$. 

We prove the proposition by induction. Since $W^{(1)} \to \Spec (\widehat{\mathcal{O}})$ is finite and surjective, there exists a monic irreducible polynomial $f_1 (y_1) \in \widehat{\mathcal{O}}[y_1]$ of $y_1$-degree $\geq 1$ that defines $W^{(1)}$. Since its intersection with the face $\{y_1 = 0\}$ is empty, the constant term of $f_1 (y_1)$ is a unit $c$ in $\widehat{\mathcal{O}}^{\times}$. Replacing $f_1 $ by $c^{-1} f_1$, we may assume that the constant term of $f_1 (y_1)$ is $1$. This shows (1), (2), (3) for $i=1$. 

Let $n \geq 2$ and let $1 \leq i < n$. Suppose we constructed $f_1 (y_1), f_2 (y_1, y_2), \cdots, f_i (y_1, \cdots, y_i)$ that define $W^{(i)} \in z^i _{\widehat{\mathfrak{m}}} (\widehat{\mathcal{O}}, i)^c$, with the conditions (1), (2), (3). In particular, $\widehat{\mathcal{O}} \to \widehat{\mathcal{O}}[y_1, \cdots, y_i]/(f_1, \cdots, f_i)$ is a finite ring extension. Since $W^{(i+1)} \to W^{(i)}$ is finite surjective, there exists a monic irreducible polynomial in the ring $ (\widehat{\mathcal{O}}[y_1, \cdots, y_i]/(f_1, \cdots, f_i)) [ y_{i+1}]$ in $y_{i+1}$ of $y_{i+1}$-degree $\geq 1$ that defines $W^{(i+1)}$. Choose any lifting of this polynomial in $\widehat{\mathcal{O}}[y_1, \cdots, y_{i+1}]$ such that the coefficient of the highest $y_{i+1}$-degree term does not involve any variable other than $y_{i+1}$. Call it $f_{i+1} (y_1, \cdots, y_{i+1})$. This thus satisfies (1) and (2) by construction. 

If the constant term of $f_{i+1} (y_1, \cdots, y_{i+1})$ is a unit $c$ in $\widehat{\mathcal{O}} ^{\times}$, then replace $f_{i+1}$ by $c^{-1} f_{i+1}$. If the constant term of $f_{i+1} (y_1, \cdots, y_{i+1})$ is not a unit in $\widehat{\mathcal{O}}^{\times}$, then it is divisible by $t \in \widehat{\mathcal{O}}$. Then first replace $f_i$ by $f_i + f_1$. This procedure does not disturb the triangular shape of \eqref{eqn:triangular O}, nor (1) or (2), and now the constant term of the new $f_i$ is a unit $c$ in $\widehat{\mathcal{O}} ^{\times}$, because any element of the form $1 + th \in \widehat{\mathcal{O}}$ is a unit. Replacing $f_{i+1}$ by $c^{-1} f_{i+1}$, we thus make it satisfies (3). 

Hence by induction, we have the triangular shaped generators $f_1, \cdots, f_n$ as in \eqref{eqn:triangular O} satisfying (1), (2) and (3). 
\end{proof}

\begin{cor}\label{cor:perturb ci}
Let $W \in z^n_{\widehat{\mathfrak{m}}} (\widehat{\mathcal{O}}, n)^c$ be a nonempty integral cycle. For a defining set $\{ f_1, \cdots, f_n \} \subseteq \widehat{\mathcal{O}}[y_1, \cdots, y_n]$ of $W$ in Proposition \ref{prop:ci}, consider the corresponding coefficient perturbation $V \subset \mathbb{A}^M_{\widehat{\mathcal{O}}} \times_{\widehat{\mathcal{O}}} \square_{\widehat{\mathcal{O}}} ^n$ given by $\{ F_1, \cdots, F_n \}$ of $\{ f_1, \cdots, f_n \}$ as in Definition \ref{defn:coeff perturb}. Then the codimension of $V$ in $ \mathbb{A}^M_{\widehat{\mathcal{O}}} \times_{\widehat{\mathcal{O}}} \square_{\widehat{\mathcal{O}}} ^n$ is $n$. 
\end{cor}

\begin{proof}
For $1 \leq i \leq n$, let $V_i$ be the closed subscheme of $\mathbb{A}_{\widehat{\mathcal{O}}} ^M \times_{\widehat{\mathcal{O}}} \square_{\widehat{\mathcal{O}}} ^i$ given by $(F_1, \cdots, F_i)$. We prove that the codimension of $V_i$ in $B_i:= \mathbb{A}_{\widehat{\mathcal{O}}} ^M \times_{\widehat{\mathcal{O}}} \square_{\widehat{\mathcal{O}}} ^i$ is $i$ by induction on $i$. 

When $i=1$, this is obvious because $V_1$ is given by a single polynomial $F_1 (y_1)$ and $\deg_{y_1} F_1 (y_1) \geq 1$, so that $F_1 (y_1) \not = 0$. Suppose the statement holds for $i \geq 1$. Then $V_i \subset B_i$ has codimension $i$ so that $V_i \times_{\widehat{\mathcal{O}}} \square_{\widehat{\mathcal{O}}} ^1 \subset B_i \times_{\widehat{\mathcal{O}}} \square_{\widehat{\mathcal{O}}} ^1 = B_{i+1}$ has codimension $i$. On the other hand, $V_{i+1}$ is given in $V_i \times_{\widehat{\mathcal{O}}}  \square_{\widehat{\mathcal{O}}} ^1$ by $F_{i+1} (y_1, \cdots, y_{i+1})$, and $\deg _{y_{i+1}} F_{i+1} \geq 1$, so that the codimension of $V_{i+1}$ in $V_i \times_{\widehat{\mathcal{O}}} \square_{\widehat{\mathcal{O}}} ^1$ is $1$. Hence the codimension of $V_{i+1}$ in $B_{i+1}$ is $i+1$, thus by induction the statement holds for all $1 \leq i \leq n$. 
\end{proof}

 \subsection{Perturbation lemmas}\label{subsec:perturb}

 We discuss several perturbation lemmas that play essential roles in the proof of the mod $t^m$-moving lemma in \S \ref{sec:mod t^m moving}.

 \subsubsection{Non-emptiness of fibers}

 Here is the basic set-up we consider:

\begin{quote}\textbf{Situation $(\star)$: } Let $W \in z^n_{\widehat{\mathfrak{m}}} (\widehat{\mathcal{O}}, n)^c $ be a nonempty integral cycle 
and choose a triangular generating set $\{ f_1, \cdots, f_n \} \subset \widehat{\mathcal{O}}[y_1, \cdots, y_n]$ of the form \eqref{eqn:triangular O} using Proposition \ref{prop:ci}. Consider the coefficient perturbation $V$ of $W$ with respect to $\{f_1, \cdots, f_n \}$ given by 
$$
(F_1, \cdots, F_n ) \subset \widehat{\mathcal{O}}[x_1, \cdots, x_M][y_1, \cdots, y_n]
$$ as in Definition \ref{defn:coeff perturb}. Let $\alpha_0 \in \widehat{\mathcal{O}} ^M$ be the coefficient vector corresponding to the generating set $\{ f_1, \cdots, f_n \}$ of $W$. By Lemma \ref{lem:cycle proper}, $W$ is closed in $ \ov{\square}_{\widehat{\mathcal{O}}}^n$. We regard $y_i = Y_{i1}/Y_{i0}$ and use $((Y_{10}; Y_{11}), \cdots, (Y_{n0}; Y_{n1})) \in \ov{\square}_{\widehat{\mathcal{O}}} ^n$ as the projective coordinates. By homogenizing each $f_j$, we obtain $\bar{f}_j \in \widehat{\mathcal{O}}[\{ Y_{10}, Y_{11}\}, \cdots, \{ Y_{n0}, Y_{n1} \}]$. 
Here $\ov{W} =W$ in $\ov{\square}_{\widehat{\mathcal{O}}} ^n$ is given by the ideal $(\bar{f}_1, \cdots, \bar{f}_n)$. Similarly, the homogenization $(\bar{F}_1, \cdots, \bar{F}_n)$ of $(F_1, \cdots, F_n)$ defines the Zariski closure $\ov{V}\subseteq \mathbb{A}_{\widehat{\mathcal{O}}} ^M \times_{\widehat{\mathcal{O}}} \ov{\square}_{\widehat{\mathcal{O}}} ^n$ of $V$. 

Let $\ov{pr}: \ov{V} \to \mathbb{A}^M_{\widehat{\mathcal{O}}}$ and $pr: V \to \mathbb{A}^M _{\widehat{\mathcal{O}}}$ be the restrictions of the projections $\mathbb{A}^M_{\widehat{\mathcal{O}}} \times \ov{\square}_{\widehat{\mathcal{O}}} ^n \to \mathbb{A}^M _{\widehat{\mathcal{O}}}$ and $\mathbb{A}^M_{\widehat{\mathcal{O}}} \times {\square}_{\widehat{\mathcal{O}}} ^n \to \mathbb{A}^M _{\widehat{\mathcal{O}}}$, respectively. For each $\alpha \in \mathbb{A}_{\widehat{\mathcal{O}}} ^M$, let $\ov{V}_{\alpha}:= \ov{pr} ^{-1} (\alpha)$. We have $V_{\alpha} = \ov{V}_{\alpha} \cap \square_{\widehat{\mathcal{O}}} ^n = pr^{-1} (\alpha)$, while $\ov{V}_{\alpha_0} = \ov{W}= W = V_{\alpha_0} $. 
\end{quote}

\begin{prop}\label{prop:nonempty}
Under the Situation $(\star)$, there exists a nonempty open neighborhood $U_{\rm ne} \subset \mathbb{A} ^M_{\widehat{\mathcal{O}}}$ of $\alpha_0$ such that for each $\alpha \in U_{\rm ne}$, the fiber $V_{\alpha}$ is nonempty. Furthermore, this open set contains $\mathbb{G}_{m, \widehat{\mathcal{O}}} ^M$.
\end{prop}

\begin{proof}\label{eqn:triangular X}
By Proposition \ref{prop:ci}, the coefficient perturbation $V$ is given by polynomials
$$
\tuborg  F_1 (y_1) \in \widehat{\mathcal{O}}[x_1, \cdots, x_M][y_1],\\
F_2 (y_1, y_2) \in \widehat{\mathcal{O}}[x_1, \cdots, x_M][y_1, y_2], \\
\ \ \ \ \ \vdots \\
F_n (y_1, \cdots, y_n) \in \widehat{\mathcal{O}}[x_1, \cdots, x_M][y_1, \cdots, y_n].\sluttuborg
$$
Let $K= {\rm Frac} (\widehat{\mathcal{O}}) = k((t))$. Here $\deg_{y_1} F_1 \geq 1$ and the coefficient in $\widehat{\mathcal{O}}[x_1, \cdots, x_M]$ of the highest $y_1$-degree term is a variable $x_{\ell_1}$ for some $1 \leq \ell_1 \leq M$. For the open subset $U_1 \subseteq \mathbb{A}_{\widehat{\mathcal{O}}} ^M$ given by $x_{\ell_1} \not = 0$, $y_1$ is algebraic over $K(x_1, \cdots, x_M)$, and there is a solution $y_1$ in an algebraic extension of $K(x_1, \cdots, x_M)$. Plug this solution $y_1$ into the second equation. Since $\deg_{y_2} F_2 \geq 1$, and the coefficient of the highest $y_2$-degree term is $x_{\ell_2}$ for some $1 \leq \ell_2 \leq M$ with $\ell_2 \not = \ell_1$. Thus, for the open set $U_2 \subseteq \mathbb{A}_{\widehat{\mathcal{O}}} ^M$ given by $x_{\ell_1} \not = 0$ and $x_{\ell_2} \not = 0$, $y_2$ is algebraic over $K(x_1, \cdots, x_M)$, and in particular there is a solution $y_2$ in an algebraic extension of $K(x_1, \cdots, x_M)$. Continuing this way, the coefficient of the highest $y_n$-degree term of $F_n$ is $x_{\ell_n}$ for some $1 \leq \ell_n \leq M$ with $\ell_n \not = \ell_1, \cdots, \ell_{n-1}$, and for the open set $U_n \subseteq \mathbb{A}_{\widehat{\mathcal{O}}} ^M$ given by $\{ x_{\ell_1} \not = 0, \cdots, x_{\ell_n} \not = 0 \}$ we have a system of solutions $y_1, \cdots, y_n$ in an algebraic extension of $K(x_1, \cdots, x_M)$. In other words, for each $\alpha \in U_{\rm ne}:= U_n$, the fiber $V_{\alpha}$ is nonempty. By construction $U_{\rm ne}$ is given by the product of $\mathbb{A}^1 _{\widehat{\mathcal{O}}}$ for each $x_i$ with $i \not \in \{ \ell_1, \cdots, \ell_n \}$ and $\mathbb{G}_{m, \widehat{\mathcal{O}}}$ for each $x_i$ with $i \in \{ \ell_1, \cdots, \ell_n \}$, so that the second statement follows. That $\alpha_0 \in U_{\rm ne}$ follows immediately. 
\end{proof}

\subsubsection{Properness over $\widehat{\mathcal{O}}$}

\begin{lem}\label{lem:small properness} 
We are under the Situation $(\star)$. Then there exists an open neighborhood $U_{\rm pr} \subseteq \mathbb{A}_{\widehat{\mathcal{O}}} ^M$ of $\alpha_0$ such that $V_{\alpha}$ is a  proper  scheme over $\Spec (\widehat{\mathcal{O}})$ for each $\alpha \in U_{\rm pr}$.
\end{lem}

\begin{proof}
Let $F^{\infty}$ be the divisor associated to $\ov{\square}_{\widehat{\mathcal{O}}} ^n \setminus \square_{\widehat{\mathcal{O}}}^n$. By Lemma \ref{lem:properness iff}, to make $V_{\alpha}$ proper over $\Spec (\widehat{\mathcal{O}})$, it is enough to require that $\ov{V}_{\alpha} \cap F^{\infty} = \emptyset$. 
Here $F^{\infty} = \sum_{i=1} ^n \{ y_{i} = \infty \} = \sum_{i=1} ^n \{ Y_{i0} = 0 \}$ so that $\ov{V}_{\alpha} \cap F^{\infty} = \emptyset$ if and only if $\ov{V}_{\alpha} \cap \{ Y_{i0} = 0 \}= \emptyset,$ for all $1 \leq i \leq n$. Recall we have $y_i = Y_{i1}/Y_{i0}$ for the projective coordinate $(Y_{i0}; Y_{i1}) \in \ov{\square}_{\widehat{\mathcal{O}}} = \mathbb{P}_{\widehat{\mathcal{O}}} ^1$.

To see which open subset of $\mathbb{A}_{\widehat{\mathcal{O}}} ^M$ would do this job, note that
the scheme $\ov{V}_{\alpha}$ \emph{does} intersect $\{ Y_{i0} = 0 \}$ if and only if the scheme given by $\{ \bar{F}_1, \cdots, \bar{F}_n, Y_{i0} \}$ has a point over $\alpha$. The system $\{ \bar{F}_1, \cdots, \bar{F}_n, Y_{i0} \}$ defines a closed subscheme of $\mathbb{A}_{\widehat{\mathcal{O}}} ^M \times_{\widehat{\mathcal{O}}} \ov{\square}_{\widehat{\mathcal{O}}} ^n$ of dimension $\leq M+n +1 - (n+1)  = M$ by Corollary \ref{cor:perturb ci}. Thus its image $C_i$ under the projective morphism $\mathbb{A}_{\widehat{\mathcal{O}}} ^M \times_{\widehat{\mathcal{O}}} \ov{\square}_{\widehat{\mathcal{O}}} ^n \to \mathbb{A}_{\widehat{\mathcal{O}}} ^M$ is a closed subscheme of dimension $\leq M$, thus $C_i \subsetneq \mathbb{A}_{\widehat{\mathcal{O}}} ^M$ is a proper closed subscheme. Hence $\ov{V}_{\alpha}$ \emph{does not} intersect with $F^{\infty}$ if and only if $\alpha \in U_{\rm pr}:= \bigcap_{i=1} ^n (\mathbb{A}_{\widehat{\mathcal{O}}} ^M \setminus C_i )$. By construction, we have $\alpha_0 \in U_{\rm pr}$. This proves the lemma.
\end{proof}

\begin{cor}\label{cor:properness}
Under the Situation $(\star)$, for every sufficiently large integer $N > 0$, the open ball $\mathcal{B}_N (\alpha_0) \subseteq k[[t]]^{M}$ in the non-archimedean $t$-adic sup-norm satisfies $(1)$ $\mathcal{B}_N (\alpha_0) \cap (k[t] ^{M})$ is nonempty, $(2)$ for every $\alpha \in \mathcal{B}_N (\alpha_0) \cap (k[t] ^{M})$, the closed subscheme $V_{\alpha}$ is proper over $Spec (\widehat{\mathcal{O}})$, and $(3)$ these so obtained polynomials $f_{1, \alpha}, \cdots, f_{n, \alpha} \in k[t][y_1, \cdots, y_n] \subseteq \mathcal{O}[y_1, \cdots, y_n]$ of $V_{\alpha}$ satisfy $f_{j, \alpha} \equiv f_{j} \mod t^m,$ for each $1 \leq j \leq n$.
\end{cor}

\begin{proof}

Since the induced non-archimedean $t$-adic topology is finer than the Zariski topology on $\mathbb{A}_{\widehat{\mathcal{O}}} ^{M}$ and $\alpha_0 \in \widehat{\mathcal{O}}^{M} = k[[t]]^{M}$, for every sufficiently large integer $N>0$, the open ball $\mathcal{B}_N (\alpha_0) \subseteq k[[t]] ^{M}$ of radius $e^{-N}$ centered at $\alpha_0$ is contained in the open subset $U_{\rm pr} \subset \mathbb{A}_{\widehat{\mathcal{O}}} ^{M-r}$ of Lemma \ref{lem:small properness}. We may assume $N>m$. But $k[t]^{M} \subseteq k[[t]]^{M}$ is dense in the non-archimedean topology, so $\mathcal{B}_N (\alpha_0) \cap (k[t]^{M})\not = \emptyset$, proving (1). Since $\mathcal{B}_N (\alpha_0) \subseteq U_{\rm pr}$, we have (2). On the other hand, $\alpha \in \mathcal{B}_N (\alpha_0)$ $\Leftrightarrow $ $|\alpha - \alpha_0| < e^{-N}$ $\Leftrightarrow$ for each $1 \leq j \leq n$, we have $f_{j, \alpha} \equiv f_j \mod t^N$. In particular, since $N > m$, $f_{j, \alpha} \equiv f_j \mod t^m$, proving (3).
\end{proof}

 \subsubsection{Flat stratum}\label{sec:flat stratum}

Let $pr: V \to \mathbb{A}_{\widehat{\mathcal{O}}} ^M$ be the restriction of the projection $\mathbb{A}_{\widehat{\mathcal{O}}} ^M \times_{\widehat{\mathcal{O}}} \square_{\widehat{\mathcal{O}}} ^n$. By Proposition \ref{prop:nonempty}, we know that the restriction $pr_{U_{\rm ne}}: pr^{-1} (U_{\rm ne}) \to U_{\rm ne}$ is surjective, but we do not know whether this is flat. By the generic flatness theorem of \cite[Th\'eor\`eme (6.9.1), p.153]{EGA4-2}, there is a nonempty open subset of $U_{\rm ne}$ over which $pr_{U_{\rm ne}}$ is flat, but this theorem does not tell us whether this open set contains $\alpha_0$. This causes an inconvenience. On the other hand, by the flattening stratification theorem of \cite[Corollaire (6.9.3), p.154]{EGA4-2}, we do know that there is a stratification partition $\{S_i\}$ of $U_{\rm ne}$ by locally closed subsets such the restriction of $pr$ over the inverse image of each $S_i$ is flat, and some stratum $S_{i_0}$ must contain $\alpha_0$. We will construct explicitly in Lemma \ref{lem:F-flat} a locally closed subset of $U_{\rm ne}$ containing $\alpha_0$ over which a more general collection of coherent sheaves are flat. This result will be used in \S \ref{sec:dom} and \S \ref{sec:gi}.

Here is the set-up updated from Situation $(\star)$:

\begin{quote}\textbf{Situation $(\star')$: } Let $W \in z^n_{\widehat{\mathfrak{m}}} (\widehat{\mathcal{O}}, n)^c$ be a nonempty integral cycle, and choose a triangular generating set $\{ f_1, \cdots, f_n \} \subset \widehat{\mathcal{O}}[y_1, \cdots, y_n]$ of the form \eqref{eqn:triangular O} using Proposition \ref{prop:ci}. Let $V$ be the coefficient perturbation of $W$ given by $\{ F_1, \cdots, F_n\}  \subset \widehat{\mathcal{O}}[x_1, \cdots, x_M][y_1, \cdots, y_n]$ as in Situation $(\star)$. By renaming the variables $x_i$, we may assume that $x_{M-n+1}, \cdots, x_M$ corresponds to the constant terms ($=1$) of $f_1, \cdots, f_n$. 
By Lemma \ref{lem:cycle proper}, $W$ is closed in $ \ov{\square}_{\widehat{\mathcal{O}}}^n$ and it is given by $(\bar{f}_1, \cdots, \bar{f}_n)$ as in Situation $(\star)$, with its coefficient perturbation $\ov{V} \subseteq\mathbb{A}_{\widehat{\mathcal{O}}} ^M \times_{\widehat{\mathcal{O}}} \ov{\square}_{\widehat{\mathcal{O}}} ^n$ given by $(\bar{F}_1, \cdots, \bar{F}_n)$. 

Let $B:= \mathbb{A}^{M-n}_{\widehat{\mathcal{O}}} \times  \mathbf{1} \subset \mathbb{A}^{M}_{\widehat{\mathcal{O}}}$, where $\mathbf{1} \subset \mathbb{A}^n_{\widehat{\mathcal{O}}}$ is the closed subscheme defined by $x_j = 1$ for $M-n+1 \leq j \leq M$, and $\ov{pr}_B: \ov{pr}^{-1} (B) \to B$ and $pr_B: pr ^{-1} (B) \to B$ are the restrictions of $\ov{pr}$ and $pr$, respectively. Here, $\alpha_0 = \beta_0 \times \mathbf{1} \in B$.
\end{quote}

\begin{lem}\label{lem:F-flat}
Under the Situation $(\star')$, denote $\mathbb{A}^M_{\widehat{\mathcal{O}}} \times_{\widehat{\mathcal{O}}}\ov{\square}_{\widehat{\mathcal{O}}}^n$ by $X.$ For each face $\ov{F} \subseteq \ov{\square}_{\widehat{\mathcal{O}}} ^n$, including the case $\ov{F} = \ov{\square}_{\widehat{\mathcal{O}}} ^n$, consider the coherent sheaf $\mathcal{F}_{\ov{F}}:= \mathcal{O}_X/ ( \mathcal{I}_{\ov{V}} + \mathcal{I}_{\ov{F}})$, where $\mathcal{I}_{\ov{V}}$ is the ideal sheaf of $\ov{V} \subseteq X$ and $\mathcal{I}_{\ov{F}}$ is the pull-back to $X$ of the ideal sheaf of $\ov{F}$. 
Then $\mathcal{F}_{\ov{F}}$ restricted to $\ov{pr}^{-1} (B)$ is $\ov{pr}_B$-flat. In particular, its restriction to $pr^{-1} (B)$ is $pr_B$-flat as well.
\end{lem}

\begin{proof} Fix a face $\ov{F},$ and denote $\mathcal{F}_{\ov{F}}$ by $\mathcal{F}$. Let $X':= B \times_{\widehat{\mathcal{O}}} \ov{\square}_{\widehat{\mathcal{O}}} ^n=\ov{pr}^{-1} (B)$, which is closed in $X$. Let $\mathcal{F}'$ be the restriction of $\mathcal{F}$ to $X'$. For each open chart $U' \subseteq X'$ from an affine cover of $X'$ and each $x \in U'$, we need to show that the stalk $\mathcal{F}_x '$ is a flat $\mathcal{O}_{B, pr_B (x)}$-module. We prove it for the chart $U':= B \times_{\widehat{\mathcal{O}}} \square_{\widehat{\mathcal{O}}} ^n$ of $X'$, which is obtained from the open chart $U: = \mathbb{A}_{\widehat{\mathcal{O}}} ^M \times_{\widehat{\mathcal{O}}} \square_{\widehat{\mathcal{O}}} ^n$ of $X$ via $U' = U \cap X'$.

Now, $\mathcal{F} |_U = \mathcal{O}_U/ (\mathcal{I}_V + \mathcal{I}_F)$ is given by the quotient of $\widehat{\mathcal{O}}[x_1, \cdots, x_M][y_1, \cdots, y_n]$ by $(F_1, \cdots, F_n) + (\{ y_{i_1}= \epsilon_1, \cdots, y_{i_s} = \epsilon_s\})$, where $\{ y_{i_1}= \epsilon_1, \cdots, y_{i_s}=  \epsilon_s\}$ for some $\epsilon_j \in \{ 0, 1 \}$, is the set of equations of the face ${F}= \ov{F} \cap \square_{\widehat{\mathcal{O}}} ^n$.

Recall the constant term of each of $f_1, \cdots, f_n$ is $1$. By the labeling convention of the Situation $(\star')$, $x_{M-n+j}$ is the variable corresponding to the nonzero constant term of $f_j$ for $1 \leq j \leq n$. So, we have $F_j = x_{M-n+j} + G_j$ for some $G_j \in \widehat{\mathcal{O}}[ x_1, \cdots, x_{M-n}][y_1, \cdots, y_n]$. Hence, the sections $(\mathcal{O}_U/ \mathcal{I}_V)  (U) = \widehat{\mathcal{O}}[ x_1, \cdots, x_M][y_1, \cdots, y_n]/ (F_1, \cdots, F_n)$ can be obtained from $\widehat{\mathcal{O}}[ x_1, \cdots, x_M][y_1, \cdots, y_n]$ by replacing each $x_{M-n+j}$ by $- G_j$  for $1 \leq j \leq n$. Here each $G_j$ does not involve any of the variables $x_{M-n+1}, \cdots, x_{M}$. Thus, $(\mathcal{O}_U/ \mathcal{I}_V)  (U) \simeq$
$$ \widehat{\mathcal{O}}[ x_1, \cdots, x_{M-n}, -G_1, \cdots, -G_n][y_1, \cdots, y_n] = \widehat{\mathcal{O}}[x_1, \cdots, x_{M-n}][y_1, \cdots, y_n],$$
which is a polynomial ring over $\widehat{\mathcal{O}}$ with the variables $\{ x_1, \cdots, x_{M-n}\} \cup \{ y_1, \cdots, y_n\}$. Now the further quotient
$$
R_F:= \widehat{\mathcal{O}}[x_1, \cdots, x_M][y_1, \cdots, y_n]/ ( (F_1, \cdots, F_n) + (\{ y_{i_1}= \epsilon_1, \cdots, y_{i_s} = \epsilon_s\}))
$$
can be obtained from $ (\mathcal{O}_U/ \mathcal{I}_V)  (U) \simeq  \widehat{\mathcal{O}}[x_1, \cdots, x_{M-n}][y_1, \cdots, y_n]$ by replacing each variable $y_{i_u}$ by $\epsilon_u$ for $1 \leq u \leq s$, i.e.
$$R_F \simeq  \widehat{\mathcal{O}}[x_1, \cdots, x_{M-n}][y_1, \cdots, y_n] / (\{ y_{i_1}= \epsilon_1, \cdots, y_{i_s} = \epsilon_s\}) $$
$$\simeq \widehat{\mathcal{O}}[ x_1, \cdots, x_{M-n}][ \{ y_\ell \ | \ 1 \leq \ell \leq n, \ell \not = i_1, \cdots, i_s\}],$$ which is again a polynomial ring over $\widehat{\mathcal{O}}$ with the variables $\{ x_1 \cdots, x_{M-n} \} \cup \{ y_\ell \ | \  1 \leq \ell \leq n, \ \ell \not  =  i_1, \cdots, i_s \}$. In particular, the natural map $\widehat{\mathcal{O}}[x_1, \cdots, x_{M-n}] \to R_F$ induced by the projection $\ov{pr}$ is injective and it is flat. Here, we have $\Spec (\widehat{\mathcal{O}}[x_1, \cdots, x_{M-n}]) = \mathbb{A}_{\widehat{\mathcal{O}}} ^{M-n} \simeq  \mathbb{A}_{\widehat{\mathcal{O}}} ^{M-n} \times \mathbf{1} = B$. Hence in particular, $\mathcal{F}'= \mathcal{F}|_{U'}$ is flat. The proof for other charts of $X'$ is similar, so we omit it.
\end{proof}

\subsubsection{Dominance}\label{sec:dom}

\begin{lem}\label{lem:small dominant}
Under the Situation $(\star')$, recall that $W \to \Spec (\widehat{\mathcal{O}})$ is dominant. Then there is an open neighborhood $U_{\rm dom} \subseteq \mathbb{A}_{\widehat{\mathcal{O}}} ^{M-r}$ of $\beta_0 $ such that for every $\beta \in U_{\rm dom}$, the associated closed subscheme $V_{\alpha} \subseteq \square_{\widehat{\mathcal{O}}}^n$ with $\alpha:= \beta \times \mathbf{1}$, is dominant over $\Spec (\widehat{\mathcal{O}})$ as well. 
\end{lem}

\begin{proof}
Let $K:= {\rm Frac}( \widehat{\mathcal{O}}) = k((t))$. Note that a morphism $Z \to \Spec (\widehat{\mathcal{O}})$ is dominant if and only if the base change $Z_K \to \Spec (K)$ is a nonempty $K$-scheme. So, we consider the situation after the base change via $\Spec (K) \to \Spec (\widehat{\mathcal{O}})$.

By Lemma \ref{lem:F-flat}, the morphism $pr_B: pr^{-1} (B) \to B$ is flat. For the open set $U_{\rm ne}$ of Proposition \ref{prop:nonempty}, we have $\alpha_0 \in B \cap U_{\rm ne}$ so that $B \cap U_{\rm ne} \not = \emptyset$, and this proposition shows that the restriction $pr_{B\cap U_{\rm ne}} : pr^{-1} (B \cap U_{\rm ne}) \to B \cap U_{\rm ne}$ is flat and surjective. Since $U_{\rm ne}$ contains $\mathbb{G}_m ^M$, there exists a nonempty open neighborhood $U' \subset \mathbb{A}_{\widehat{\mathcal{O}}} ^{M-n}$ of $\beta_0$ such that $U' \times \mathbf{1} \subseteq B \cap U_{\rm ne}$. Hence $pr_{U'} : pr^{-1} (U' \times \mathbf{1}) \to U' \times \mathbf{1}$ is flat and surjective. So, after base change via $\Spec (K) \to \Spec (\widehat{\mathcal{O}})$, the new morphism $pr_{U'_K} : pr^{-1} (U' \times \mathbf{1})_K \to (U' \times \mathbf{1})_K$ is flat and surjective. We implicitly used \cite[Proposition (2.1.4), p.6]{EGA4-2} several times. For this flat family, the dimensions of the fibers are all equal (see \cite[Corollaire (6.1.2), p.135]{EGA4-2}, or \cite[Corollary III-9.6, p.257]{Hartshorne}). In particular, for every $\beta \in \mathbb{A}_K ^{M-n} \cap U' $, we have $ 0 \leq \dim ({V}_{\alpha_0})  = \dim ({V}_{\alpha})$ with $\alpha=\beta \times \mathbf{1}$. In particular $V_{\alpha} \not = \emptyset$. But, $\mathbb{A}_K ^{M-n}$ is a nonempty open subset of $\mathbb{A}_{\widehat{\mathcal{O}}} ^{M-n}$, so that we can take $U_{\rm dom} := \mathbb{A}^{M-n} _K \cap U'$ to finish the proof of the lemma.
\end{proof}

\begin{cor}\label{cor:small dominant}
Under the assumptions of Lemma \ref{lem:small dominant}, for every sufficiently large integer $N>0$, the open ball $\mathcal{B}_N (\beta_0) \subseteq k[[t]]^{M-n}$ in the non-archimedean $t$-adic sup-norm satisfies $(1)$ $\mathcal{B}_N (\beta_0) \cap (k[t]^{M-n})$ is nonempty, $(2)$ for every $\beta \in \mathcal{B}_N (\beta_0) \cap (k[t]^{M-n})$, the closed subscheme $V_{\alpha}$ for $\alpha = \beta \times \mathbf{1}$, is dominant over $\Spec (\widehat{\mathcal{O}})$, and $(3)$ these so obtained polynomials $f_{1, \alpha}, \cdots, f_{n, \alpha} \in k[t][y_1, \cdots, y_n] \subseteq \mathcal{O}[y_1, \cdots, y_n]$ of $V_{\alpha}$ satisfy $f_{j, \alpha} \equiv f_{j} \mod t^m$ for each $1 \leq j \leq n$. 
\end{cor}

\begin{proof}The proof is almost identical to that of Corollary \ref{cor:properness}, where we use Lemma \ref{lem:small dominant} instead of Lemma \ref{lem:small properness}, so we omit it. 
\end{proof}

\subsubsection{Geometric integrality}\label{sec:gi}

Although we began with an integral scheme $W$, this integrality may not necessarily be preserved under ``small" perturbations of the coefficients. However, we will show that the geometrical integrality over $k$ in the sense of \cite[D\'efinition (4.6.2), p.68]{EGA4-2} is better behaved. Later in Case 2 of the proof of Theorem \ref{thm:mod t^m}, we will reduce the general integral situation to the geometrically integral situation.

\begin{lem}\label{lem:small int}
Under the Situation $(\star')$, suppose  further  that $W$ is geometrically integral over $k$. Then there exists an open neighborhood $U_{\rm gi} \subseteq \mathbb{A}_{\widehat{\mathcal{O}}}^{M-n}$ of $\beta_0$ such that for each $\beta \in U_{\rm gi}$, the fiber $V_{\alpha}$ with $\alpha = \beta \times \mathbf{1}$, is geometrically integral over $k.$ 
\end{lem}

\begin{proof}
Note that $V_{\alpha}$ is geometrically integral over if and only if so is its Zariski closure $\ov{V}_{\alpha}$ in $\ov{\square}_{\widehat{\mathcal{O}}}^n$. Now, by Lemma \ref{lem:F-flat} with $F=\ov{\square}_{\widehat{\mathcal{O}}} ^n$, the morphism $\ov{pr}_B: \ov{pr}^{-1} (B) \to B = \mathbb{A}_{\widehat{\mathcal{O}}}^{M-n} \times \mathbf{1}$ is proper and flat. Hence by \cite[Th\'eor\`eme (12.2.4)-(viii), p.183]{EGA4-3}, the set $U_{\rm gi}:= \{ \beta \in \mathbb{A}_{\widehat{\mathcal{O}}}^{M-n} \ |  \ov{V}_{\alpha} \mbox{ with $\alpha= \beta \times \mathbf{1}$, is geometrically integral} \}$ is open in $\mathbb{A}_{\widehat{\mathcal{O}}} ^{M-n}$. This $U_{\rm gi}$ is nonempty because $\beta_0 \in U_{\rm gi}$. But again, for each $\beta \in U_{\rm gi}$ with $\alpha= \beta \times \mathbf{1}$, we have that $\ov{V}_{\alpha}$ is geometrically integral if and only if so is $V_{\alpha}$. This proves the lemma.
\end{proof}

\begin{cor}\label{cor:perturb ci restr}
Under the assumptions of Lemma \ref{lem:small int}, the restriction $V_B$ is geometrically integral of codimension $n$ in $B \times _{\widehat{\mathcal{O}}} \square_{\widehat{\mathcal{O}}} ^n$, and it intersects each codimension $1$ face of $B \times_{\widehat{\mathcal{O}}} \square_{\widehat{\mathcal{O}}} ^n$ properly.
\end{cor}

\begin{proof}
For the generic point $\eta \in \mathbb{A}_{\widehat{\mathcal{O}}} ^{M-n} \simeq B$, we have $\eta \times \mathbf{1} \in U_{\rm gi} \times \mathbf{1}$ for the nonempty open subset $U_{\rm gi} \subset \mathbb{A}_{\widehat{\mathcal{O}}} ^{M-n}$ of Lemma \ref{lem:small int}. Hence $V_{\eta \times \mathbf{1}}$ is geometrically integral. But this is equivalent to that its Zariski closure $V_{B}$ in $B \times_{\widehat{\mathcal{O}}} \square_{\widehat{\mathcal{O}}} ^n$ is geometrically integral. That it has codimension $n$ in  $B \times_{\widehat{\mathcal{O}}} \square_{\widehat{\mathcal{O}}} ^n$ follows by the same argument as in the proof of Corollary \ref{cor:perturb ci}.

Finally, let $F \subset B \times_{\widehat{\mathcal{O}}} \square_{\widehat{\mathcal{O}}} ^n$ be a codimension $1$ face given by $\{ y_i = \epsilon \}$ for some $1 \leq i \leq n$ and $\epsilon \in \{ 0, 1 \}$. Since this is a divisor and since $V_{B}$ is geometrically integral (in particular irreducible) by the previous paragraph, we just need to show that $V \not \subseteq \mathbb{A}_{\widehat{\mathcal{O}}} ^{M-n} \times _{\widehat{\mathcal{O}}} F$. Suppose not, i.e. $V_B \subseteq \mathbb{A}_{\widehat{\mathcal{O}}} ^{M-n} \times_{\widehat{\mathcal{O}}} F$. Then specializing at $ \beta_0 \times \mathbf{1} \in B$, which corresponds to the given $W$, we have $W \subseteq F$. But this is impossible because $W$ intersects $F$ properly, in fact $W \cap F = \emptyset$ by Corollary \ref{cor:no face}. This is a contradiction, so $V_B$ intersects each codimension $1$ face properly.
\end{proof}

\begin{cor}\label{cor:small int}
Under the assumptions of Lemma \ref{lem:small int}, for every sufficiently large integer $N>0$, the open ball $\mathcal{B}_N (\beta_0) \subseteq k[[t]]^{M-n}$ in the non-archimedean $t$-adic sup-norm satisfies $(1)$ $\mathcal{B}_N (\beta_0) \cap (k[t]^{M-n})$ is nonempty, $(2)$ for every $\beta \in \mathcal{B}_N (\beta_0) \cap (k[t]^{M-n})$ and $\alpha:= \beta \times \mathbf{1}$, the closed subscheme $V_{\alpha}$ is geometrically integral, and $(3)$ these so obtained polynomials $f_{1, \alpha}, \cdots, f_{n, \alpha} \in k[t][y_1, \cdots, y_n] \subseteq \mathcal{O}[y_1, \cdots, y_n]$ of $V_{\alpha}$ satisfy $f_{j, \alpha} \equiv f_{j} \mod t^m$ for each $1 \leq j \leq n$.
\end{cor}

\begin{proof}
The proof is almost identical to that of Corollary \ref{cor:properness}, where we use Lemma \ref{lem:small int} instead of Lemma \ref{lem:small properness}, so we omit it.
\end{proof}

 \subsubsection{Empty intersection with faces}
 
 Recall from Corollary \ref{cor:no face} that for any proper face $F \subsetneq \square_{\widehat{\mathcal{O}}} ^n$, we had $W \cap F = \emptyset$, which is stronger than having proper intersection with the face. We assert that this is an open condition inside a suitable locally closed base in the following sense:

\begin{lem}\label{lem:small proper int}
Under the assumptions of Lemma \ref{lem:small int}, for each proper face $F \subsetneq \square_{\widehat{\mathcal{O}}} ^n$, there exists an open neighborhood $U_F \subseteq \mathbb{A}_{\widehat{\mathcal{O}}} ^{M-n}$ of $\beta_0$ such that for each $\beta \in U_F$, we have $V_{\alpha} \cap F= \emptyset$ with $\alpha = \beta \times \mathbf{1}$. In particular, for each $\beta \in U_{\rm pi}:= \bigcap_F U_F \subseteq \mathbb{A}_{\widehat{\mathcal{O}}} ^M$, where the intersection is taken over all proper faces $F$, the closed subscheme $V_{\alpha}$ intersects with no proper face at all.
\end{lem}

\begin{proof}
By Lemma \ref{lem:properness iff}, we know that $W \cap F = \emptyset$ if and only if $\ov{W} \cap \ov{F} = \emptyset$. So, we want to achieve the stronger assertion that $\ov{V}_{\alpha} \cap \ov{F} = \emptyset$ for each $\alpha = \beta \times \mathbf{1}$, where $\beta$ is in an open neighborhood of $\beta_0$ in $B$. We use the projectivized system $\{ \bar{F}_1, \cdots, \bar{F}_n \}$ of equations for $\ov{V}$ so that $\{ \bar{F}_1, \cdots, \bar{F}_n, x_{M-n+1}-1 , \cdots, x_M-1 \}$ is the system for $\ov{V}_B$. 

 When $\ov{F}$ is a codimension $1$ face of $\ov{\square}_{\widehat{\mathcal{O}}} ^n$, it is given by $\{y_{i_1} = \epsilon_1\}=\{ Y_{i_1, 1} = \epsilon_1 Y_{i_1, 0}\}$ for some $1 \leq i_1 \leq n$ and $\epsilon_1 \in \{ 0, 1 \}$. Here, the scheme $\ov{V}_{\alpha}$ \emph{does} intersect with the face $\ov{F}$ if and only if the scheme given by $\{ \bar{F}_1, \cdots, \bar{F}_n, x_{M-n+1} -1, \cdots, x_M -1 ,  Y_{i_1,1} - \epsilon_1 Y_{i_1,0}\}$ has a point lying over $\alpha$. Here, the system $\{ \bar{F}_1, \cdots, \bar{F}_n, x_{M-n+1}-1, \cdots, x_M -1,  Y_{i_1,1} - \epsilon_1 Y_{i_1, 0} \}$ defines a closed subscheme of $\mathbb{A}^{M-n}_{\widehat{\mathcal{O}}} \times_{\widehat{\mathcal{O}}} \ov{\square}_{\widehat{\mathcal{O}}} ^n$ of dimension $ \leq \dim \ ( \mathbb{A}_{\widehat{\mathcal{O}}} ^{M-n} \times_{\widehat{\mathcal{O}}} \ov{\square}_{\widehat{\mathcal{O}}} ^n) - (n+1) = (M-n) + n +1 - (n + 1) = M-n$ by Corollary \ref{cor:perturb ci restr}. Thus its image $C_F$ under the projective morphism $\mathbb{A}^{M-n}_{\widehat{\mathcal{O}}} \times_{\widehat{\mathcal{O}}} \ov{\square}_{\widehat{\mathcal{O}}} ^n \to \mathbb{A}^{M-n}_{\widehat{\mathcal{O}}}$ is a closed subscheme of dimension $\leq M-n$. In particular, $C_F \subsetneq \mathbb{A}_{\widehat{\mathcal{O}}} ^{M-n}$ is a proper closed subscheme since $\dim \ (\mathbb{A}_{\widehat{\mathcal{O}}} ^{M-n} )= M-n+1$. Hence $\ov{V}_{\alpha}$ \emph{does not} intersect with $\ov{F}$ if and only if $\alpha \in U_F:= \mathbb{A}_{\widehat{\mathcal{O}}} ^{M-n} \setminus C_F$. By construction we have $\alpha_0 \in U_F$. Here, $\ov{V}_{\alpha} \cap \ov{F} = \emptyset$ implies that $V_{\alpha} \cap F = \emptyset$.
 
Since every proper face is contained in some codimension $1$ face, this proves the lemma.
\end{proof}

\begin{cor}\label{cor:small proper int}
Under the assumptions of Lemma \ref{lem:small int}, for every sufficiently large integer $N > 0$, the open ball $\mathcal{B}_N (\alpha_0) \subseteq k[[t]]^{M}$ in the non-archimedean $t$-adic sup-norm satisfies $(1)$ $\mathcal{B}_N (\alpha_0) \cap (k[t] ^{M})$ is nonempty, $(2)$ for every $\alpha \in \mathcal{B}_N (\alpha_0) \cap (k[t] ^{M})$, the closed subscheme $V_{\alpha}$ does not intersect any face $F\subsetneq \overline{\square}_{\widehat{\mathcal{O}}}^n$ at all, and $(3)$ these so obtained polynomials $f_{1, \alpha}, \cdots, f_{n, \alpha} \in k[t][y_1, \cdots, y_n] \subseteq \mathcal{O}[y_1, \cdots, y_n]$ of $V_{\alpha}$ satisfy $f_{j, \alpha} \equiv f_{j} \mod t^m,$ for each $1 \leq j \leq n$.
\end{cor}

\begin{proof}
The proof is almost identical to that of Corollary \ref{cor:properness}, where we use Lemma \ref{lem:small proper int} instead of Lemma \ref{lem:small properness}, so we omit it. 
\end{proof}

\subsection{The mod $t^m$ moving lemmas}\label{sec:mod t^m moving}

First observe:

\begin{lem}\label{lem:flat family}
Let $T$ be a ${\rm Spec} (\widehat{\mathcal{O}})$-scheme of finite type. Let $W_1, W_2 \subseteq T$ be two integral closed subschemes, both surjective over $\Spec (\widehat{\mathcal{O}})$, such that we have the equality $W_{1,s} = W_{2,s}$ of the special fibers. Then $\dim \ W_1 = \dim \ W_2$.
\end{lem}

\begin{proof}
Let $d_i:= \dim \ W_i$. The morphisms $W_i \to \Spec (\widehat{\mathcal{O}})$ for $i=1,2$ are flat (of relative dimension $d_i -1$) because they are surjective and $\Spec (\widehat{\mathcal{O}})$ is a regular scheme of dimension $1$ (see \cite[Corollaire (6.1.2), p.135]{EGA4-2} or \cite[Proposition III-9.7, p.257]{Hartshorne}). Since $W_{1,s} = W_{2,s}$, we have $d_1 -1 = d_2 -1$. Hence $d_1= d_2$.
\end{proof}

We now prove the main result of \S \ref{sec:perturbation}:

\begin{thm}\label{thm:mod t^m}
For the completion homomorphism $\xi^n: z^n_{\mathfrak{m}} (\mathcal{O}, n)^c \to z^n_{\widehat{\mathfrak{m}}} (\widehat{\mathcal{O}}, n)^c$, the composition $\xi_m^n : z^n_{\mathfrak{m}} (\mathcal{O}, n) ^c \overset{\xi^n}{\to} z^n _{\widehat{\mathfrak{m}}}(\widehat{\mathcal{O}}, n)^c \to  z^n (k_m, n)$ is a surjection. 
\end{thm}

\begin{proof}
Let $W \in z^n _{\widehat{\mathfrak{m}}} (\widehat{\mathcal{O}}, n)^c$ be a nonempty integral closed subscheme of $\square_{\widehat{\mathcal{O}}} ^n$. 
By Lemma \ref{lem:cycle proper}, the structure morphism $W \to \Spec (\widehat{\mathcal{O}})$ is surjective.

\textbf{Case 1:} First consider the case when $W$ is geometrically integral over $k.$ Take the generators of the ideal of $W$ given by $f_1, \cdots, f_n  \in \widehat{\mathcal{O}}[y_1, \cdots, y_n]$ satisfying the Situation $(\star')$, i.e. of the form in \eqref{eqn:triangular O} in Proposition \ref{prop:ci}.

By Corollaries \ref{cor:properness}, \ref{cor:small dominant}, \ref{cor:small int} and \ref{cor:small proper int}, there exists a sufficiently large integer $N>m$ such that for every $\beta \in \mathcal{B}_N (\beta_0) \cap (k[t])^{M-n}$ with $\alpha := \beta \times \mathbf{1}$, the corresponding cycle $V_{\alpha} \subseteq \square_{\widehat{\mathcal{O}}}^n$ is proper and dominant (in particular surjective) over $\Spec (\widehat{\mathcal{O}})$, is geometrically integral, and has empty intersection with all proper faces of $\square_{\widehat{\mathcal{O}}}^n$ (in particular, the intersections with all faces are proper), and furthermore the defining ideal of $V_{\alpha}$ in $ \widehat{\mathcal{O}}[y_1, \cdots, y_n]$ is given by polynomials $f_{j, \alpha} \in k[t][y_1, \cdots, y_n]$ satisfying $f_j \equiv f_{j, \alpha} \mod t^m$ for all $1 \leq j \leq n$. Since both $W$ and $V_{\alpha}$ are geometrically integral over $k$, they are integral, thus we have $W \sim_{t^m} V_{\alpha}$. By Lemma \ref{lem:flat family}, this implies $\dim \ W=\dim \ V_{\alpha}$. Furthermore, for each proper face $F \subsetneq \square_{\widehat{\mathcal{O}}} ^n$, we have $V_{\alpha} \cap F_s = \emptyset$ so that $\codim_{F} (V_{\alpha} \cap F_s)  \geq n$, while for $F= \square_{\widehat{\mathcal{O}}} ^n$ we have $\codim_{F} (V_{\alpha} \cap F_s) = \codim_{F} ((V_{\alpha})_s ) = \codim_F (W_s )  \geq n$, so that $V_{\alpha} \in z^n_{\widehat{\mathfrak{m}}}(\widehat{\mathcal{O}}, n)^c$. 

Note that $V_{\alpha}$ is given by the ideal generated by $\{ f_{1, \alpha}, \cdots, f_{n, \alpha} \}$ in $\widehat{\mathcal{O}}[y_1, \cdots, y_n]$ with $f_{j, \alpha} \in k[t][y_1, \cdots, y_n] \subseteq \mathcal{O}[y_1, \cdots, y_n]$. So, if we let $Z \subseteq \square_{\mathcal{O}}^n$ be the closed subscheme given by the ideal generated by the same set $ \{ f_{1, \alpha}, \cdots, f_{n, \alpha}\}$, this time in $\mathcal{O}[y_1, \cdots, y_n]$, then we have $\widehat{Z}:= Z \times_{\mathcal{O}} \widehat{\mathcal{O}} = V_{\alpha}$ by definition. 

We claim that $Z \in z^n_{\mathfrak{m}} (\mathcal{O}, n)^c$. Here for each face $F_{\mathcal{O}} \subseteq \square_{\mathcal{O}} ^n$, we have $\dim \ (Z \cap F_{\mathcal{O}}) = \dim \  (\widehat{Z} \cap F_{\widehat{\mathcal{O}}})$, where $F_{\widehat{\mathcal{O}}}$ is the base change of $F_{\mathcal{O}}$. In particular, when $F_{\mathcal{O}} = \square_{\mathcal{O}}^n$, we have $\dim \ Z = \dim \ V_{\alpha}$, while when $F \subsetneq \square_{\mathcal{O}}^n$ is a proper face, $Z$ intersects with $F$ properly. Furthermore via the identification $\mathcal{O}/\mathfrak{m} = \widehat{\mathcal{O}}/\widehat{\mathfrak{m}}$, we have $Z_s = (V_{\alpha})_s = W_s$ so that $\codim_{F_{\mathcal{O}}} (Z \cap F_{\mathcal{O},s}) \geq n$ for each face $F_{\mathcal{O}} \subseteq \square_{\mathcal{O}}^n$. Hence $Z \in z^n_{\mathfrak{m}} (\mathcal{O}, n)$. The structure morphism $Z \to \Spec (\mathcal{O})$ is proper by \cite[Proposition (2.7.1)-(vii), p.29]{EGA4-2}, because its base change $Z \times_{\mathcal{O}} \widehat{\mathcal{O}} = V_{\alpha} \to \Spec (\widehat{\mathcal{O}})$ via the faithfully flat morphism $\Spec (\widehat{\mathcal{O}}) \to \Spec (\mathcal{O})$, is proper. Hence $Z \in z^n_{\mathfrak{m}} (\mathcal{O}, n)^c$ and $V_{\alpha} = \xi^n (Z)$. Combined with that $W \sim_{t^m} V_{\alpha}$, we thus have $W \in {\rm im} (\xi^n _m)$.

\textbf{Case 2:} Now we suppose that $W$ is integral, but not geometrically integral over $k$. Recall from \cite[(4.3.1), p.58]{EGA4-2} that a field extension $k \subset K$ is a \emph{primary extension} if the biggest algebraic separable extension of $k$ in $K$ is $k$ itself. In other words, we say that $k$ is separably closed in $K$, or that the separable closure of $k$ in $K$ is itself. Here we have:

\textbf{Claim:} \emph{Let $w \in W$ be the generic point and take $K:= k(w)$. Let $L$ be the algebraic separable closure of $k$ in $K$. Then we have the Cartesian diagram with a section $s$ of $p_2$:
$$
\xymatrix{ \Spec (L \otimes_k K) \ar[r] _{\ \ \ p_2} \ar[d] _{p_1} & \Spec (K)  \ar@/_1.3pc/[l] _{s} \ar[d] \\
\Spec (L) \ar[r]_{p_{L/k}} & \Spec (k),}
$$
for the \'etale base change map $p_{L/k}$. In other words, $p_2$ is a Nisnevich cover.}

Since $k \subset L$ is separable, i.e. \'etale, the base change $p_2$ is \'etale as well. So, we just need to prove the existence of the section $s$. Recall that $p_2$ is given by the $k$-algebra homomorphism $p_2 ^{\sharp}: K \to L \otimes_k K$ given by $a \mapsto 1 \otimes a$. Note that we have the multiplication map $\Delta^{\sharp}: K \otimes_k K \to K$. Let $s^{\sharp}:= \Delta^{\sharp}|: L \otimes_k K \to K$ be the restriction of $\Delta^{\sharp}$ on the subalgebra $L \otimes_k K \subset K \otimes_k K$. Then the composite
$$
s^{\sharp} \circ p_2 ^{\sharp} : K \to L \otimes_k K \to K,  \ \ \ a \mapsto 1 \otimes a \mapsto a
$$
is indeed the identity of $K$ so that we have $p_2 \circ s = {\rm Id}_{\Spec (K)}$. This proves the Claim.

\bigskip

Going back to the proof of the Theorem in Case 2, we continue to use the notations of Claim and let $x:= s(w)$. Note we have $k(x) \simeq k(w)$. The flat pull-back $p_{L/k} ^*:  z^n_{\widehat{\mathfrak{m}}} (\widehat{\mathcal{O}}, n)^c \to z^n_{\widehat{\mathfrak{m}}'} (\widehat{\mathcal{O}}_{L}, n)^c$ of Proposition \ref{prop:fpb} gives $p_{L/k} ^* (W) = \sum_{i} m_i W_i '$ for some integers $m_i \geq 1$ and integral closed subschemes $W_i ' \subset \square_{\widehat{\mathcal{O}}_L} ^n$. Here, $\widehat{\mathcal{O}}_L:= \widehat{\mathcal{O}}_{\mathbb{A}_{L}^1, 0} \simeq L[[t]]$ and $\widehat{\mathfrak{m}}' := (t) \subset L[[t]]$.

By the Claim, one of $W_i'$ has $x= s(w)$ as its generic point, say $W_{i_0} '$. Hence $k(W_{i_0}') = k(x) = K$. On the other hand, this field is a primary extension of $L$ by the choice of $L$ that it is the algebraic separable closure of $k$ in $K$. Now, \cite[Proposition (4.5.9), p.62]{EGA4-2} says that this is equivalent to that $W_{i_0}'$ is geometrically irreducible in the sense of \cite[D\'efinition (4.5.2), p.61]{EGA4-2}.

If $k$ is perfect (thus so is $L$), then geometric reducedness over $L$ (see \cite[D\'efinition (4.6.2), p.68]{EGA4-2}) is equivalent to reducedness by \cite[Corollaire (4.6.11), p.70]{EGA4-2}, so that this $W_{i_0}'$ is actually geometrically integral over $L$. In case $k$ is not perfect, then by \cite[Proposition (4.6.6), p.69]{EGA4-2}, there exists a finite radicial (i.e. purely inseparable) extension $L \subset L'$ such that for the further base change of $W_{i_0}'$ to $(W_{i_0}')_{L'}$ from $L$ to $L'$, the scheme $((W_{i_0}')_{L'})_{\rm red}$ is geometrically reduced over $L'$. Under this procedure, $((W_{i_0}')_{L'})_{\rm red}$ is still geometrically irreducible over $L'$ by definition. If $x'$ is its unique generic point, then $k(x') \simeq k(x) \simeq k(w)$ under this extension. Hence under the further base change $\Spec (L') \to \Spec (k)$ (which is no longer \'etale, so the map $\Spec (L' \otimes_k K) \to \Spec (K)$ may no longer be a Nisnevich cover), but it still has a section $s' : \Spec (K) \to \Spec (L' \otimes_k K)$ such that the irreducible component of $x'=s'(w)$ is geometrically integral over $L'$, and $k(s'(w)) \simeq k(w)$, i.e. its degree to $W$ is $1$. In this case, we replace $L$ by $L'$ and $W_{i_0}'$ by $((W_{i_0}')_{L'})_{\rm red}$.

Hence for any base field $k$, we have a geometrically irreducible $L$-scheme $W_{i_0}' \in z^n_{\widehat{\mathfrak{m}}'} (\widehat{\mathcal{O}}_{L}, n)^c$ such that $[k(W_{i_0}') : k(W) ] = [K: K] = 1$. In particular, the push-forward $p_{L/k, *} (W_{i_0}') = W$.

We have the following commutative diagram:
\begin{equation}\label{eqn:mod t^m cd}
\xymatrix{ z^n_{\mathfrak{m}'} (\mathcal{O}_L, n)^c \ar[r] ^{\xi^n_L} \ar[d] ^{p_{L/k, *}} & z^n _{\widehat{\mathfrak{m}}'} (\widehat{\mathcal{O}}_L, n)^c \ar[d] ^{p_{L/k, *}} \ar[r] &  z^n (L_m, n) \ar[d] ^{p_{L/k, *}}\\
z^n _{\mathfrak{m}} (\mathcal{O}, n)^c \ar[r] ^{\xi^n} \ar@/_1.5pc/[rr]_{\xi^n_m} & z^n _{\widehat{\mathfrak{m}}} (\widehat{\mathcal{O}}, n)^c \ar[r] &  z^n (k_m, n) ,}
\end{equation}
where the left square is commutative by \cite[Proposition 1.7]{Fulton}, while the right square is well-defined and commutative by Proposition \ref{prop:ppf}-(2) applied to the proper morphism $p_{L/k}: \Spec (L) \to \Spec (k)$.

Since $W'_{i_0}$ is geometrically integral over $L$, by Case 1, there exists some $Z' \in z^n _{\mathfrak{m}'} (\mathcal{O}_L, n)^c$ such that $\xi^n_{L} (Z')  = \widehat{Z}' \sim_{t^m} W'_{i_0}$. Hence we have $W = p_{L/k, *} (W'_{i_0}) \sim_{t^m} ^{\dagger} p_{L/k, *}  \xi ^n_{L} (Z')=^{\ddagger} \xi ^n  p_{L/k, *} (Z')$, where $\dagger$ and $\ddagger$ hold by the commutativity of the right and the left squares of the diagram \eqref{eqn:mod t^m cd}, respectively. This shows that $W$ lies in the image of $\xi^n_m$. This finishes the proof that $\xi^n_m$ is surjective.
\end{proof}

\begin{cor}\label{cor:Milnor surj}
The morphism $\xi ^n _m: z^n_{\mathfrak{m}} (\mathcal{O}, \bullet)^{pc} \to z^n (k_m, \bullet)$ of complexes 
induces a surjective group homomorphism $\CH_{\mathfrak{m}} ^n (\mathcal{O}, n)^{pc} \to \CH^n (k_m, n)$.
\end{cor}

\begin{proof}
Let $K_{\bullet}:= \ker (\xi^n_m)$ and $I_{\bullet}:= {\rm im} (\xi ^n_m)$ so that we have a short exact sequence $0 \to K_{\bullet} \to  z^n_{\mathfrak{m}} (\mathcal{O}, \bullet)^{pc} \to I_{\bullet} \to 0$ of homological complexes. From the morphisms $z^n_{\mathfrak{m}} (\mathcal{O}, \bullet)^{pc} \to I_{\bullet} \hookrightarrow z^n (k_m, \bullet)$ of complexes, we have homomorphisms
\begin{equation}\label{eqn:Milnor surj}
\CH^n_{\mathfrak{m}} (\mathcal{O}, n)^{pc} \to {\rm H_n} (I_{\bullet}) \to \CH^n (k_m, n).
\end{equation}
Here, by Remark 
 \ref{remk:0-cycle trivial 2}, we have $z^{n}_{\mathfrak{m}} (\mathcal{O}, n-1)^{pc} = 0$ 
 so that $K_{n-1} = 0$, while we have $K_j= 0$ for all $j \leq n-1$ due to the dimension reason. In particular, ${\rm H}_{n-1} (K_{\bullet})  = 0$ and we have part of the associated long exact sequence $\cdots \to \CH^n_{\mathfrak{m}} (\mathcal{O}, n)^{pc} \to {\rm H}_n (I_{\bullet}) \to {\rm H}_{n-1} (K_{\bullet}) = 0$ so that the first map $\CH_{\mathfrak{m}} ^n (\mathcal{O}, n)^{pc} \to {\rm H_n} (I_{\bullet})$ of \eqref{eqn:Milnor surj} is surjective. 

On the other hand, by Theorem \ref{thm:mod t^m}, we have $I_n = z^n (k_m, n)$, while $I_j = 0$ for all $j \leq n-1$ by Remark \ref{remk:0-cycle trivial cor} and the dimension reason. Hence
$${\rm H}_n (I_{\bullet}) = \frac{ z^n (k_m, n)}{ \partial ( \xi^n_m  (z^n_{\mathfrak{m}} (\mathcal{O}, n+1)^{pc}))},  \ \ \CH^n (k_m, n) = \frac{z^n (k_m, n)}{ \partial ( z^n  (k_m, n+1))}$$
 with $\partial ( \xi^n_m ( z^n_{\mathfrak{m}} (\mathcal{O}, n+1)^{pc})) \subseteq \partial (z^n (k_m, n+1))$ in $z^n (k_m, n)$, so that the second map  ${\rm H}_n (I_{\bullet}) \to \CH^n (k_m, n)$ in \eqref{eqn:Milnor surj} is the surjective quotient map. Hence the composite in \eqref{eqn:Milnor surj} is surjective, as desired.
\end{proof}

 \begin{remk}
 One may wonder whether Theorem \ref{thm:mod t^m} extends beyond the Milnor range, i.e. when $q <n$, whether the composite $z^q _{\mathfrak{m}} (\mathcal{O}, n)^{pc} \to z^q _{\widehat{\mathfrak{m}}} (\widehat{\mathcal{O}}, n)^{pc} \to z^q (k_m, n)$ is surjective. To test if this question is affirmatively answerable, concentrate only on the subset of integral effective cycles. Since the cycles considered are flat over ${\rm Spec} (\widehat{\mathcal{O}})$, such effective cycles may be, under mild additional assumptions, represented by (a locally closed subset of) a Hilbert scheme $H$, and there exists a (non-constant) morphism $\Spec (\widehat{\mathcal{O}}) \to H$ of schemes. On the other hand, if the surjectivity assertion mod $t^m$ would hold for those integral effective cycles, then it implies that for the fpqc cover $\Spec (\widehat{\mathcal{O}}) \to \Spec (\mathcal{O})$, the morphism $\Spec (\widehat{\mathcal{O}}) \to H$ should give an fpqc descent to a morphism $\Spec (\mathcal{O}) \to H$. However, this means that there exists a non-constant rational map $\mathbb{A}^1 \dashrightarrow H$, which imposes a restrictive condition on $H$. Thus, we do not expect an extension of Theorem \ref{thm:mod t^m} to cycles of arbitrary dimension.
 \end{remk}

\section{Milnor range II: mod $t^m$-equivalence and conclusion}\label{sec:Milnor 2}

In \S \ref{sec:graph}, we use the mod $t^m$ moving lemma of Theorem \ref{thm:mod t^m} to transport the main theorem of \cite[Theorem 3.4]{EVMS} (or equivalently, \cite{Gabber}, \cite{KMS}) to our situation of cycles over $\widehat{\mathcal{O}}$ modulo $t^m$. This allows a significant simplification of the generators of our relative cycle group $\CH^n ((k_m, (t)), n)$, and helps in finally proving in \S \ref{sec:gr mod t^m} that the regulators $\Upsilon_i$ defined in Proposition \ref{prop:main reciprocity} of \S \ref{sec:Milnor} respect the mod $t^m$-equivalence. Using this, the proof of Theorem \ref{thm:Milnor} is finished in \S \ref{sec:final proof}.

 \subsection{The graph cycles}\label{sec:graph}
 
 Recall that for each integral $k$-domain $R$ of finite Krull dimension, and a sequence $a_1, \cdots, a_n \in R^{\times}$ of units, we have its associated closed subscheme $\Gamma_{(a_1, \cdots, a_n)} \subset \square_R ^n$ given by the set of equations $\{y_1 = a_1, \cdots, y_n = a_n \}.$ This is called the \emph{graph cycle} of the sequence, and this is geometrically integral over $k$. In case $R$ is local with the maximal ideal $\mathfrak{m}$, actually $\Gamma_{(a_1, \cdots, a_n)} \in z^n_{\mathfrak{m}} (R, n)$, and we get the graph homomorphism $gr: K_n ^M (R) \to \CH^n_{\mathfrak{m}} (R, n)$. This was proven in \cite[Lemma 2.1]{EVMS} for a ring $R$ essentially of finite type over $k$, but exactly the same argument proves it for the general case. By construction, the Zariski closure $\ov{\Gamma}$ of $\Gamma$ in $\ov{\square}_R ^n$ is equal to $\Gamma$, so that in particular $\Gamma$ is closed in $\ov{\square}_R ^n$ as well. Furthermore, one sees immediately that $\partial_i ^{\epsilon} (\Gamma) = 0$ for each $1 \leq i \leq n$ and $\epsilon\in \{ 0, \infty\}$. 
We improve \emph{loc.cit.} a bit as follows:
 
 \begin{lem}\label{lem:im_gr}
The graph homomorphisms $gr_{\mathcal{O}}: K_n ^M (\mathcal{O}) \to \CH_{\mathfrak{m}} ^n (\mathcal{O}, n)$ and $gr_{\widehat{\mathcal{O}}} : K_n ^M (\widehat{\mathcal{O}}) \to \CH_{\widehat{\mathfrak{m}}} ^n (\widehat{\mathcal{O}}, n)$ of \cite[Lemma 2.1]{EVMS} actually map into $\CH^n_{\mathfrak{m}} (\widehat{\mathcal{O}}, n)^{pc}$ and $\CH_{\widehat{\mathfrak{m}}} ^n (\widehat{\mathcal{O}}, n)^{pc}$, respectively.
 \end{lem}
 
 \begin{proof}
We give the proof for $\mathcal{O}$ only. The proof for $\widehat{\mathcal{O}}$ is identical. First note that when $a_1, \cdots, a_n \in \mathcal{O}^{\times}$, the graph $\Gamma_{(a_1, \cdots, a_n)}$ is already proper over $\Spec (\mathcal{O})$ with no intersection with the faces, and the proper intersection condition with respect to the special fiber. Thus we have the set map $gr_{\mathcal{O}}: (\mathcal{O}^{\times})^n \to z^n_{\mathfrak{m}} (\mathcal{O}, n)^c$.

The point of the rest of the proof is to repeat part of the argument of \cite[Lemma 2.1]{EVMS} and check that the relevant cycles used in \emph{loc.cit.} that put various relations on $z^n_{\mathfrak{m}} (\mathcal{O}, n)^c$ are actually lying in the group $z^n_{\mathfrak{m}} (\mathcal{O}, n+1)^{pc}$. 
\begin{enumerate}
\item Let $a, 1-a, a_i \in \mathcal{O}^{\times}$ for $3 \leq i \leq n$. Consider the parametrized cycle
$$
W: \square_{\mathcal{O}} ^1 \dashrightarrow \square_{\mathcal{O}} ^{n+1}, \ \  x \mapsto \left( x, 1-x, \frac{ a-x}{1-x} , a_3, \cdots, a_n \right).
$$
To check $W \in z^n_{\mathfrak{m}} (\mathcal{O}, n+1)^{pc}$, we need to look at its faces. By a direct calculation, one checks that the only nonzero face is $\partial_3 ^0 W = ( a, 1-a, a_3, \cdots, a_n)$, which is in $z^n_{\mathfrak{m}} (\mathcal{O}, n)^c$, so $\partial W \in z^n_{\mathfrak{m}} (\mathcal{O}, n)^c$. Hence we have $W \in z^n_{\mathfrak{m}} (\mathcal{O}, n+1)^{pc}$ by definition. This also shows that $\{ a, 1-a, a_3, \cdots, a_n \} \mapsto 0$ in $\CH^n _{\mathfrak{m}} (\mathcal{O}, n)^{pc}$. 

\item Let $a, b, a_i \in \mathcal{O} ^{\times}$ for $2 \leq i \leq n$. Consider the parametrized cycle
$$
W: \square_{\mathcal{O}} ^1 \dashrightarrow \square_{\mathcal{O}} ^{n+1}, \ \ x \mapsto \left( x, \frac{ a x - ab}{ x- ab}, a_2, \cdots, a_n \right).
$$
By direct calculations, its only nontrivial faces are $\partial_1 ^{\infty} W = (a, a_2, \cdots, a_n)$, $\partial_2 ^{0} W = (b, a_2, \cdots, a_n)$, $\partial_2 ^{\infty}W= ( ab, a_2, \cdots, a_n)$, all of which are in $z^n_{\mathfrak{m}} (\mathcal{O}, n)^c$. Hence $\partial W \in z^n_{\mathfrak{m}} (\mathcal{O}, n)^c$ so that $W \in z^n_{\mathfrak{m}} (\mathcal{O}, n+1)^{pc}$ by definition. This also gives the relation $(ab, a_2, \cdots, a_n) \equiv (a, a_2, \cdots, a_n) + (b, a_2, \cdots, a_n)$ in $\CH^n _{\mathfrak{m}} (\mathcal{O}, n)^{pc}$. 
\end{enumerate}
We may permute the above cycles to give the induced homomorphism $gr_{\mathcal{O}}: \bigotimes_{i=1} ^n \mathcal{O} ^{\times} \to \CH^n _{\mathfrak{m}} (\mathcal{O}, n)^{pc}$ by (2), which descends to $gr_{\mathcal{O}}: K_n ^M (\mathcal{O}) \to \CH^n _{\mathfrak{m}} (\mathcal{O}, n)^{pc}$ by (1). This completes the proof.
 \end{proof}
  
\begin{lem}\label{lem:grOsurj}
Let $k$ be an infinite field. Then the map $gr_{\mathcal{O}}: K_n ^M (\mathcal{O}) \to \CH^n_{\mathfrak{m}} (\mathcal{O}, n)^{pc}$ of Lemma \ref{lem:im_gr} is surjective.
\end{lem}

\begin{proof}
We first claim that $\CH^n_{\mathfrak{m}} (\mathcal{O}, n)^{pc} = \CH^n (\mathcal{O}, n)$. By \cite[Lemma 3.11]{EVMS}, every cycle class in $\CH^n (\mathcal{O}, n)$ is represented by a cycle in the group called $\CH^n _{\sfs} (\mathcal{O}, n)$, where each irreducible component is finite (in particular, proper) surjective over $\Spec (\mathcal{O})$ with the proper intersection condition with $\mathfrak{m}$. (See \cite{EVMS} for its precise definition.) This means that the composite inclusions $\CH^n_{\sfs} (\mathcal{O}, n) \hookrightarrow \CH^n_{\mathfrak{m}}(\mathcal{O}, n)^{pc} \hookrightarrow \CH^n_{\mathfrak{m}} (\mathcal{O}, n)$ is an isomorphism, where the second arrow is injective by Corollary \ref{cor:pc inj}. Hence $\CH^n_{\sfs} (\mathcal{O}, n) = \CH^n_{\mathfrak{m}}(\mathcal{O}, n)^{pc} =\CH^n_{\mathfrak{m}} (\mathcal{O}, n)$. Now the easy moving lemma of smooth affine $k$-schemes of higher Chow groups shows that $\CH^n_{\mathfrak{m}} (\mathcal{O}, n) = \CH^n (\mathcal{O}, n)$. This proves the claim.

Now the graph map $K_n ^M (\mathcal{O}) \to \CH^n (\mathcal{O}, n)$ is surjective by \cite[Theorem 3.4]{EVMS} (which uses that $k$ is infinite), and it factors through $\CH^n_{\mathfrak{m}} (\mathcal{O}, n)^{pc}$ by Lemma \ref{lem:im_gr}, so via the equality $\CH^n_{\mathfrak{m}} (\mathcal{O}, n)^{pc} = \CH^n (\mathcal{O}, n)$, the map $gr_{\mathcal{O}}$ of Lemma \ref{lem:im_gr} is surjective.
\end{proof}

 \begin{defn}
 For $R= \mathcal{O}$ or $\widehat{\mathcal{O}}$, let $z_{gr} ^n (R, n)$ be the subgroup generated by the images of the graph cycles $\Gamma_{(a_1, \cdots, a_n)}$ over all sequences $a_1, \cdots, a_n \in R^{\times}$.  
 For the well-defined homomorphism $z^n_{gr} (\widehat{\mathcal{O}}, n) \to \CH^n_{\widehat{\mathfrak{m}}} (\widehat{\mathcal{O}}, n)^{pc} \to \CH^n (k_m, n)$, define $\CH^n_{gr} (k_m, n)$ to be the image of $z^n _{gr}(\widehat{\mathcal{O}},n)$ in $\CH^n (k_m, n)$. 
  \end{defn} 
  
 \begin{lem}\label{lem:gr surj}
 Let $k$ be an infinite field. The composite $ K_n ^M (\widehat{\mathcal{O}}) \overset{gr_{\widehat{\mathcal{O}}}}{\to} \CH^n_{\widehat{\mathfrak{m}}} (\widehat{\mathcal{O}},n)^{pc}  \to \CH^n (k_m, n)$ is surjective, where $gr_{\widehat{\mathcal{O}}}$ is as in Lemma \ref{lem:im_gr}. In particular, the group $\CH^n (k_m, n)$ is generated by the graph cycles $\Gamma_{(a_1, \cdots, a_n)}$ for sequences $a_1, \cdots, a_n \in \widehat{\mathcal{O}} ^{\times}$, and the natural homomorphism $\CH^n_{gr} (k_m, n)\to \CH^n (k_m, n)$ is an isomorphism.
  \end{lem}

\begin{proof}
We have a commutative diagram
$$
\xymatrix{ K_n ^M (\mathcal{O}) \ar[r]^{gr_{\mathcal{O}}\ \ } \ar[d] & \CH^n _{\mathfrak{m}} (\mathcal{O}, n)^{pc} \ar[d] ^{\xi} \ar[dr]^{* \ \ \ \ } & \\
K_n ^M (\widehat{\mathcal{O}} )\ar[r]^{gr_{\widehat{\mathcal{O}}}\ \ } & \CH^n _{\widehat{\mathfrak{m}}}(\widehat{\mathcal{O}}, n) ^{pc} \ar[r] & \CH^n (k_m, n),}
$$
where $gr_{\mathcal{O}}$ and $gr_{\widehat{\mathcal{O}}}$ map into their respective target groups of the diagram by Lemma \ref{lem:im_gr}, the left square commutes by \cite[Proposition 2.3]{EVMS}. The map $gr_{\mathcal{O}}$ is surjective by Lemma \ref{lem:grOsurj}. The sloped map $*$ is surjective by Corollary \ref{cor:Milnor surj}. By diagram chasing, the map $K_n ^M (\widehat{\mathcal{O}}) \to \CH^n (k_m, n)$ is surjective. The second assertion follows immediately from the first one.
\end{proof}

\begin{lem}\label{lem:old Claim} 
Let $k$ be an infinite field. The surjection $K_n ^M (\widehat{\mathcal{O}}) \to \CH^n (k_m, n)$ of Lemma \ref{lem:gr surj} induces a surjection $K_n ^M (k_m) \to \CH^n (k_m, n)$.
\end{lem}

\begin{proof}
There is a natural surjection $K_n ^M (\widehat{\mathcal{O}}) \to K_n ^M (\widehat{\mathcal{O}}/(t^m)) = K_n ^M (k_m)$. So, for any Milnor symbol $\{ a_1, \cdots, a_n \} \in K_n ^M (k_m)$ with $a_i \in k_m^{\times}$, we choose any liftings $\tilde{a}_1, \cdots, \tilde{a}_n \in \widehat{\mathcal{O}} ^{\times} = k[[t]]^{\times}$ and send the symbol $\{ \tilde{a}_1, \cdots, \tilde{a}_n \} \in K_n ^M (\widehat{\mathcal{O}})$ to the cycle class in $\CH^n (k_m, n)$ of the graph cycle $\Gamma_{(\tilde{a}_1, \cdots, \tilde{a}_n)} \subset \square_{\widehat{\mathcal{O}}} ^n$. To prove that this map is well-defined, choose another sequence of liftings $\tilde{a}_1', \cdots, \tilde{a}_n ' \in \widehat{\mathcal{O}} ^{\times}$ of the sequence $a_1, \cdots, a_n \in k_m ^{\times}$, and here $\tilde{a}_i - \tilde{a}_i ' \in t^m k[[t]]$. By definition, we have $\Gamma_{(\tilde{a}_1, \cdots, \tilde{a}_n)} \sim_{t^m} \Gamma_{(\tilde{a}_1', \cdots, \tilde{a}_n')}$, so that the map $K_n ^M (k_m) \to \CH^n (k_m, n)$ is well-defined. The surjectivity of this map now follows from the surjectivity of $K_n ^M (\widehat{\mathcal{O}}) \to \CH^n (k_m, n)$ of Lemma \ref{lem:gr surj}.
\end{proof}

\subsection{The graph cycles over $\widehat{\mathcal{O}}$ mod $t^m$}\label{sec:gr mod t^m}
 
 For graph cycles, it is easy to describe mod $t^m$ equivalence:
 \begin{lem}\label{lem:gr mod t^m}
 Let $Z_1, Z_2 \in z^n_{gr} (\widehat{\mathcal{O}}, n)^c$ be two integral graph cycles, represented by
 \begin{equation}\label{eqn:gr mod t^m}
 Z_1: \ \{ y_1 = a_1, \cdots, y_n = a_n\}, \ \ \ Z_2: \ \{y_1 = b_1, \cdots, y_n = b_n\},
 \end{equation}
 where $a_j, b_j \in \widehat{\mathcal{O}} ^{\times}$ for $1 \leq j \leq n$. Then the following are equivalent:
 \begin{enumerate}
 \item $Z_1 \sim_{t^m} Z_2$
 \item For each $1 \leq j \leq n$, we have $a_j \equiv b_j$ in $\widehat{\mathcal{O}}/ (t^m)$.
 \item For each $1 \leq j \leq n$, there exists $c_j \in \widehat{\mathcal{O}}$ such that $a_j = b_j (1+ c_j t^m)$ in $\widehat{\mathcal{O}}$.
 \end{enumerate}
 \end{lem}
 
 \begin{proof}
The equivalence (1) $\Leftrightarrow$ (2) and the implication (3) $\Rightarrow$ (2) are obvious. For the implication (2) $\Rightarrow$ (3), note that $a_j \equiv b_j$ in $\widehat{\mathcal{O}}/(t^m)$ implies that $a_j b_j ^{-1} \equiv 1$ in $\widehat{\mathcal{O}}/(t^m)$ so that $a_j b_j^{-1} = 1 + c_j t^m$ in $\widehat{\mathcal{O}}$ for some $c_j\in \widehat{\mathcal{O}}$. This proves (3).
 \end{proof}

\begin{prop}\label{prop:main mod t^m gr}
Let $k$ be a field of characteristic $0$.
Let $Z_1, Z_2 \in z^n_{gr}(\widehat{\mathcal{O}}, n)^c$ be two integral graph cycles such that ${Z}_1 \sim_{t^m} {Z}_2$. Then $\Upsilon_i (Z_1) = \Upsilon_i (Z_2)$ for each $1 \leq i \leq m-1$.
\end{prop}

\begin{proof}
For $Z_1$ and $Z_2$, express them by the equations as in \eqref{eqn:gr mod t^m}. 
By Lemma \ref{lem:gr mod t^m}, the assumption that $Z_1 \sim_{t^m} Z_2$ implies that we have $a_j = b_j (1 + c_j t^m)$ for some $c_j \in \widehat{\mathcal{O}}$ for each $1 \leq j \leq n$. Notice that the common special fiber $(Z_1)_s= (Z_2)_s$ is given by a single closed point $\mathfrak{p}$ whose coordinates are $\bar{a}_1= \bar{b}_1, \cdots, \bar{a}_n = \bar{b}_n \in k$, where the bars denote the images in the residue field $k = \widehat{\mathcal{O}}/(t)$.
For each $j$, we have 
$$d \log a_j - d \log b_j = d \log (1 + c_j t^m) = \frac{ t^m d c_j + c_j m t^{m-1} d t}{ 1 + c_j t^m} \in t^{m-1} \Omega_{\widehat{\mathcal{O}}/\mathbb{Z}} ^1.$$
Hence by expanding out $d\log y_1 \wedge \cdots \wedge d \log y_n|_{Z_1} - d \log y_1 \wedge \cdots \wedge d \log y_n |_{Z_2} = d \log a_1\wedge \cdots \wedge d\log a_n - d \log b_1 \wedge \cdots \wedge d \log b_n$, we directly check that it is in $t^{m-1} \Omega_{\widehat{\mathcal{O}}/\mathbb{Z}} ^n$. Thus for each $1 \leq i \leq m-1$, we have $  \frac{1}{t^i} d \log a_1 \wedge \cdots \wedge d \log a_n - \frac{1}{ t^i} d \log b_1 \wedge \cdots \wedge d \log b_n \in t^{m-1-i} \Omega_{\widehat{\mathcal{O}}/\mathbb{Z}} ^n \subset \Omega_{\widehat{\mathcal{O}}/\mathbb{Z}} ^n$ so that the residue at $t=0$ (which is the residue at the unique closed point $\mathfrak{p}$ of the common special fiber) of the difference vanishes. In other words, $\Upsilon_i (Z_1) = \Upsilon_i (Z_2)$ for $1 \leq i \leq n$.
\end{proof}

\begin{cor}\label{cor:main map}
Let $k$ be a field of characteristic $0$.
For $1 \leq i \leq m-1$, the map $\Upsilon_i$ of Proposition \ref{prop:main reciprocity} induces a homomorphism $\Upsilon_i: \CH^n ( (k_m, (t)), n) \to \Omega_{k/\mathbb{Z}}^{n-1}$.
\end{cor}
 
\begin{proof}
By Proposition \ref{prop:main reciprocity}, the map $\Upsilon_i$ descends to $\Upsilon_i: \CH^n _{\widehat{\mathfrak{m}}}(\widehat{\mathcal{O}}, n)^{pc} \to \Omega_{k/\mathbb{Z}} ^{n-1}$. Since $\CH^n_{gr}  (k_m, n) = \CH^n (k_m, n)$ by Lemma \ref{lem:gr surj}, we may consider only the graph cycles. For all the pairs of mod $t^m$-equivalent integral graph cycles, by Proposition \ref{prop:main mod t^m gr}, the maps $\Upsilon_i$ respect the mod $t^m$-equivalence, so that we have the induced map $\Upsilon_i : \CH^n (k_m, n) \to \Omega_{k/\mathbb{Z}} ^{n-1}$. Now, since $\CH^n (k_m, n) = \CH^n ((k_m, (t)), n) \oplus \CH^n (k, n)$, by restriction we have the desired homomorphism.
\end{proof}
 
 \begin{remk} In fact, $\Upsilon_i|_{\CH^n (k, n)} = 0$ for $1 \leq i \leq m-1$. Indeed, by the theorem of Nesterenko-Suslin \cite{NS} and Totaro \cite{Totaro}, we have an isomorphism $K_n ^M (k) \simeq \CH^n (k, n)$ so that it is enough to check that for the graph cycles $\Gamma$ given by the equations of the form $\{y_1 = a_1, \cdots, y_n = a_n\}$, with $a_1, \cdots, a_n \in k^{\times}$, we have $\Upsilon_i (\Gamma) = 0$. The form is $\frac{1}{t^i} d \log y_1 \wedge \cdots \wedge d \log y_n|_{\Gamma} = \frac{1}{t^i} d \log a_1 \wedge \cdots \wedge d \log a_n$ with each $a_j  \in k^{\times}$ so that there is no term with $dt$ anywhere in the form. Thus its residue along $t=0$ is $0$, i.e. $\Upsilon_i (\Gamma) = 0$.
 \end{remk}

\subsection{Proof of Theorem \ref{thm:Milnor}}\label{sec:final proof}
 
Finally, we prove the main theorem of the paper. 
We show that $\bigoplus_{i=1} ^{m-1} \Upsilon_i: \CH^n ((k_m, (t)), n) \to \bigoplus_{i=1} ^{m-1} \Omega_{k/\mathbb{Z}} ^{n-1}$ is an isomorphism.
Recall from Lemma \ref{lem:old Claim} that we had a surjection $K_n ^M (k_m) \to \CH^n (k_m, n)$. This induces a surjective map $K_n ^M (k_m, (t))\to \CH^n ((k_m, (t)), n)$, where $K_n ^M (k_m, (t)):= \ker ( K_n ^M (k_m) \overset{{\rm ev}_{t=0}}{\to} K_n ^M (k))$. 
We know from Proposition \ref{prop:MK rel} in the appendix \S \ref{sec:appendix} below that we have an isomorphism 
$$K^M_n (k_m, (t)) \xrightarrow{\sim} \Omega_{k_m, (t)/\mathbb{Z}} ^{n-1}/ d \Omega_{k_m, (t)/\mathbb{Z}} ^{n-2}  \xleftarrow{\sim} \bigoplus_{1 \leq i \leq m-1} t^i \Omega_{k/\mathbb{Z}} ^{n-1},$$ 
given by $\{ a_1, \cdots, a_n\} \mapsto \log (a_1) d \log (a_2) \wedge \cdots \wedge d \log (a_n)$, where $a_1 \in 1 + t k_m$ and $\Omega_{k_m, (t)/\mathbb{Z}} ^i := \ker ( \Omega_{k_m/\mathbb{Z}} ^i \overset{{\rm ev}_{t=0}}{\to} \Omega_{k/\mathbb{Z}} ^i)$. Then, looking at the $k^{\times}$-weight $i$ parts, we obtain the maps
\begin{equation}\label{eqn:main surj}
 \Omega_{k/\mathbb{Z}} ^{n-1} \xrightarrow{\sim} t^i \Omega_{k/\mathbb{Z}} ^{n-1} \hookrightarrow  K_n ^M (k_m, (t))  \twoheadrightarrow \CH^n ((k_m,  (t)), n)\overset{\Upsilon_i}{ \to } \Omega_{k/\mathbb{Z}} ^{n-1},
\end{equation}
where $r_1 d r_2 \wedge \cdots \wedge dr_n \in \Omega_{k/\mathbb{Z}} ^{n-1}$ is mapped  to $\{ e ^{ r t^i}, r_2, \cdots, r_n \} \in K_n ^M (k_m, (t))   $, where $r:= r_1 \cdots r_n$. Let $\Gamma \in z^n _{\widehat{\mathfrak{m}}} (\widehat{\mathcal{O}}, n)^c$ denote the graph of this Milnor element. The composition \eqref{eqn:main surj} then sends $r_1 dr_2 \wedge \cdots \wedge dr_n$ to $\Upsilon_i (\Gamma) = i r_1 d r_2 \wedge \cdots \wedge dr_n$ by a straightforward calculation. 
Since $i \not = 0$, the composition \eqref{eqn:main surj} is an isomorphism. In particular, the composite
\begin{equation}\label{eqn:direct main surj}
\bigoplus_{i=1} ^{m-1} \Omega_{k/\mathbb{Z}} ^{n-1} \simeq K_n ^M (k_m, (t) ) \twoheadrightarrow \CH^n ((k_m, (t)), n) \overset{\bigoplus_i \Upsilon_i}{\to} \bigoplus_{i=1} ^{m-1} \Omega_{k/\mathbb{Z}}^{n-1}
\end{equation}
is an isomorphism. Therefore, the above map $K_n ^M (k_m, (t)) \to \CH^n ((k_m, (t)), n)$ is injective, hence an isomorphism. Since the composite \eqref{eqn:direct main surj} is an isomorphism, this  implies that $\bigoplus_i \Upsilon_i$ is an isomorphism, as desired. \qed

\subsection{Appendix}\label{sec:appendix}
 In the middle of the proof of Theorem \ref{thm:Milnor} in \S \ref{sec:final proof}, we used the following Proposition \ref{prop:MK rel}. This is probably well-known to the experts, and with some effort it should follow from e.g. \cite{Goodwillie}. However, since the Milnor $K$-groups are given by the concrete Milnor symbols we sketch a direct argument as follows, partly due to the fact  that the authors could not find a suitable reference. We first have:
 
\begin{lem}\label{lem:relKgen}
Let $k$ be a field. Then $K_n ^M (k_m, (t))$ is generated by the Milnor symbols $\{ a_1, \cdots, a_n \}$ with $a_1 \in 1 + t k_m$ and $a_2, \cdots, a_n \in k_m ^{\times}$.
\end{lem}

\begin{proof}
Let $G \subseteq K_n ^M (k_m)$ be the subgroup generated by the Milnor symbols $\{ a_1, \cdots, a_n \}$ with $a_1 \in 1 + t k_m$ (which is contained in $k_m ^{\times}$) and $a_2, \cdots, a_n \in k_m ^{\times}$. Certainly under the evaluation map ${\rm ev}_{t=0} : K_n ^M (k_m) \to K_n ^M (k)$, we have $ \{ a_1 |_{t=0} , a_2 |_{t=0} , \cdots, a_n |_{t=0} \} = \{ 1, a_2|_{t=0}, \cdots, a_n |_{t=0} \}=0$ in $K_n ^M (k)$. Hence each such generator $\{a_1, \cdots, a_n \}$ with $a_1 \in 1 + t k_m$ is contained in $ \ker ({\rm ev}_{t=0}) = K_n ^M (k_m, (t)) $, thus $G \subseteq K_n ^M (k_m, (t))$.

Every $a \in k_m ^{\times}$ can be written as the product $a = c \cdot b$ with $c \in k^{\times}$ and $b \in 1 + t k_m$. Hence by the multi-linearity and the anti-commutativity of $K_n ^M (k_m)$, every symbol $\{ a_1, \cdots, a_n \}$ with $a_i \in k_m ^{\times}$ can be written as a sum of symbols in $G$ (type I) and symbols $\{ c_1, \cdots, c_n \}$ such that $c_i \in k^{\times}$ (type II). Here the splitting ring homomorphisms $k \to k_m \overset{{\rm ev}_{t=0}}{\to} k$ induce the splitting $K_n ^M (k_m ) = K_n ^M (k) \oplus K_n ^M (k_m, (t))$. The type II symbols are definitely in $K_n ^M (k)$, while the symbols of type I generate $G$. Hence $K_n ^M (k_m, (t))  = G$. 
\end{proof}

 \begin{prop}\label{prop:MK rel}
 Let $k$ be a field of characteristic $0$ and let $m \geq 2$ be an integer. Then we have an isomorphism $\phi_n: K_n ^M (k_m, (t)) \simeq \Omega_{k_m, (t)/\mathbb{Z}} ^{n-1} / d \Omega_{k_m, (t)/\mathbb{Z}} ^{n-2}$ given by $\{a_1, \cdots, a_n \} \mapsto \log (a_1) d \log (a_2) \wedge \cdots \wedge d\log (a_n)$, where $a_1 \in 1 + tk_m$, where $\log (a_1)$ makes sense in $k_m$. The isomorphism can be rewritten as $K_n ^M (k_m, (t)) \simeq \bigoplus_{i=1}^{m-1} t^i \Omega_{k/\mathbb{Z}} ^{n-1}$, where the map $t^i \Omega_{k/\mathbb{Z}} ^{n-1} \to K_n ^M (k_m, (t))$ is given by sending $r_1 dr_2 \wedge \cdots dr_n$ to $\{ e ^{r_1 r_2 \cdots  r_n t^i}, r_2, \cdots, r_n \} \in K_n ^M (k_m, (t))$.
 \end{prop}

\begin{proof}
By Lemma \ref{lem:relKgen}, $K_n ^M (k_m, (t))$ is generated by $\{ a_1, \cdots, a_n \}$ with $a_1 \in 1 + t k_m$ and $a_2, \cdots, a_n \in k_m ^{\times}$. We define $\psi_{n}: \Omega_{k_m, (t)/\mathbb{Z}} ^{n-1} / d \Omega_{k_m, (t)/\mathbb{Z}} ^{n-2} \to K_{n} ^M (k_m, (t))$ by sending $r_1 d r_2 \wedge \cdots \wedge d r_{n}$, where $r_1, \cdots, r_{\ell} \in (t)$ and $r_{\ell + 1}, \cdots, r_n \in k_m ^{\times}$, to 
$\{ e^{r_1 r_{\ell+1} \cdots r_n}, e^{r_2}, \cdots, e^ {r_\ell}, r_{\ell+1}, \cdots, r_n\}$ in $ K_n ^M (k_m, (t)).$

One can check by induction that $\phi_n$ and $\psi_n$ are well-defined group homomorphisms. We omit the proof as they follow from elementary but tedious arguments. Let's check that $\phi_n$ and $\psi_n$ are inverse to each other. Indeed, for $x= r_1 d r_2 \wedge \cdots \wedge dr_n   \in \Omega_{k_m, (t)/\mathbb{Z}} ^{n-1} / d \Omega_{k_m, (t)/\mathbb{Z}} ^{n-2} $ with $r_1, \cdots, r_{\ell} \in (t)$ and $r_{\ell + 1}, \cdots, r_n \in k_m ^{\times}$, we have
\begin{eqnarray*}
& &( \phi_{n}\circ \psi_n) (x) = \phi_n \{ e^{r_1 r_{\ell+1} \cdots r_n}, e^{r_2}, \cdots, e^ {r_\ell}, r_{\ell+1}, \cdots, r_n\}\\
&=& \log  (e^{r_1 r_{\ell+1} \cdots r_n} )d \log (e^{r_2}) \wedge \cdots \wedge d \log ( e^ {r_\ell} ) \wedge d \log (r_{\ell+1} )\wedge \cdots \wedge d \log ( r_n)\\
&=& r_1 r_{\ell+1} \cdots r_n d r_2 \wedge \cdots \wedge d r_{\ell} \wedge \frac{ d r_{\ell+1}}{r_{\ell+1}} \wedge \cdots \wedge \frac{d r_n}{r_n}= r_1 d r_2 \wedge \cdots \wedge dr_n=x,
\end{eqnarray*}
so that $\phi_n \circ \psi_n = {\rm Id}$. On the other hand, for $y= \{ a_1, a_2, \cdots, a_n \}$ with $a_1 \in 1 + t k_m$ and $a_i \in k_m ^{\times}$ for $2 \leq i \leq n$, we have $( \psi_n \circ \phi_n) (y) =$
\begin{equation}\label{eqn:relMKcomp}
 \psi_n (\log (a_1) d \log (a_2) \wedge \cdots \wedge d \log (a_n))= \psi_n \left( \frac{ \log (a_1)}{ a_2 \cdots a_n} d a_2 \wedge \cdots \wedge d a_n \right).
\end{equation}
Here $a_1 \in 1 + t k_m$ so that $\log (a_1) \in (t)$, hence $ \log (a_1)/ (a_2 \cdots a_n)\in (t)$. 
Hence \eqref{eqn:relMKcomp} equals to $\{ e ^{ \frac{ \log (a_1)}{a_2 \cdots a_n } \cdot a_2 \cdots a_n }, a_2, \cdots, a_n \} = \{ a_1, \cdots, a_n \} = y,$ i.e. $\psi_n \circ \phi_n = {\rm Id}$. The second statement follows from Lemma \ref{lem:IKW} below.
\end{proof}

 We used the following elementary lemma in the middle of the proof of Proposition \ref{prop:MK rel}, which we learned from the proof of \cite[Lemma 6.2]{IwasaKai}:

\begin{lem}\label{lem:IKW}
Let $k$ be a field of characteristic $0$. Let $m \geq 2$ be an integer. Then for $n \geq 2$, we have $\Omega_{k_m, (t)/\mathbb{Z}} ^{n-1}/ d \Omega_{k_m, (t)/ \mathbb{Z}} ^{n-2} \simeq d \Omega_{k_m/\mathbb{Z}} ^{n-1}/ d \Omega_{k/\mathbb{Z}} ^{n-1} \underset{d}{\overset{\simeq}{\leftarrow}} t k_m \otimes_k \Omega_{k/\mathbb{Z}} ^{n-1} = \bigoplus_{i=1} ^{m-1} t^i \Omega_{k/\mathbb{Z}} ^{n-1}.$
\end{lem}

\begin{proof}
 We have a commutative diagram with exact rows
$$
\xymatrix{ 
0 \ar[r] & {\rm H} ^{n-1} (\Omega_{k_m/\mathbb{Z}} ^{\bullet}) \ar[d]^{{\rm ev}_{t=0}} \ar[r] & \Omega_{k_m/\mathbb{Z}} ^{n-1}/ d \Omega_{k_m/\mathbb{Z}} ^{n-2} \ar[d] ^{{\rm ev}_{t=0}} \ar[r]^{ \ \ \ \ \ d } & d \Omega_{k_m/\mathbb{Z}} ^{n-1} \ar[d] ^{{\rm ev}_{t=0}} \ar[r] & 0 \\
0 \ar[r] & {\rm H} ^{n-1} (\Omega_{k/\mathbb{Z}}^{\bullet}) \ar[r] & \Omega_k ^{n-1} / d \Omega_{k/\mathbb{Z}} ^{n-2} \ar[r]^{ \ \ \ \ \ d } & d \Omega_{k/\mathbb{Z}} ^{n-1} \ar[r] & 0,}
$$
where the vertical maps are all split surjections. Furthermore, by the Poincar\'e lemma in \cite[Corollary 9.9.3]{Weibel}, we have ${\rm H}^{n-1} ( \Omega_{k_m, (t)/\mathbb{Z}} ^{\bullet}) = 0$ so that the left vertical map is actually an isomorphism. Hence the snake lemma gives an isomorphism $\Omega_{k_m, (t)/\mathbb{Z}} ^{n-1}/ d \Omega_{k_m, (t)/\mathbb{Z}} ^{n-2} \simeq d \Omega_{k_m/\mathbb{Z}} ^{n-1}/ d \Omega_{k/\mathbb{Z}} ^{n-1}$. The second isomorphism $d \Omega_{k_m/\mathbb{Z}} ^{n-1}/ d \Omega_{k/\mathbb{Z}} ^{n-1} \underset{d}{\overset{\simeq}{\leftarrow}} t k_m \otimes_k \Omega_{k/\mathbb{Z}} ^{n-1}= \bigoplus_{i=1} ^{m-1} t^i \Omega_{k/\mathbb{Z}} ^{n-1}$ is obvious.
\end{proof}

\subsection{Final remarks}
We have two remarks on strengthening Theorem \ref{thm:Milnor}.

\begin{remk}\label{remk:tech diff 1}
We could have defined $z^q(k_m, n)$  in Definition \ref{defn:mod t^m cx} as $z^q_{\widehat{\mathfrak{m}}} (\widehat{\mathcal{O}}, n)/ \sim_{t^m}$ using the complex $z^q_{\widehat{\mathfrak{m}}} (\widehat{\mathcal{O}}, n)$, but then part of the perturbation results in \S \ref{sec:perturbation} may not be easy to establish. If one can prove an analogue of Corollary \ref{cor:pc nb2} for $n=q$ and $\widehat{\mathcal{O}}$, i.e. the guess that ``\emph{every integral cycle $Z \in z^n_{\widehat{\mathfrak{m}}} (\widehat{\mathcal{O}}, n)$ is equivalent to a cycle in $z^n _{\widehat{\mathfrak{m}}} (\widehat{\mathcal{O}}, n)^{pc} = z^n _{\widehat{\mathfrak{m}}} (\widehat{\mathcal{O}}, n)^{c}$ modulo the boundary of a cycle in $z^n_{\widehat{\mathfrak{m}}} (\widehat{\mathcal{O}}, n+1)$}", then we can still prove a stronger version of Theorem \ref{thm:Milnor} for the cycles using $z^q_{\widehat{\mathfrak{m}}} (\widehat{\mathcal{O}}, n)/ \sim_{t^m}$. It is true when $n=1$. See Remark \ref{remk:n=q}. However for $n \geq 2$, we could yet find neither a proof nor a counterexample to the guess, so we gave this version of the definition in Definition \ref{defn:mod t^m cx}. 
\end{remk}

\begin{remk}
Reflecting on the main theorem of \cite{R}, it is desirable to remove the assumption that the base field $k$ is of characteristic $0$ in Theorem \ref{thm:Milnor}. The right hand side $(\Omega_{k/\mathbb{Z}} ^{n-1} )^{\oplus (m-1)}$ of the isomorphism of Theorem \ref{thm:Milnor} should be replaced by the big de Rham-Witt forms $\mathbb{W}_{m-1} \Omega_k ^{n-1}$ for a general base field $k$. To proceed further, we need to understand whether there exists a Parshin-Lomadze residue for the big de Rham-Witt complexes when the base field is of positive characteristic, and especially when it is imperfect. This is not trivial and may require serious works. A more minor problem  is to give an  explicit description of the relative Milnor $K$-groups of the ring of  truncated polynomials over a field of positive characteristic in terms of the big de Rham-Witt complexes. This would improve Proposition \ref{prop:MK rel} for a field of positive characteristic. We leave these as future tasks to finish.
\end{remk}

\end{document}